\documentclass[11pt,a4paper,reqno]{amsart}
\usepackage[margin=2cm]{geometry}
\usepackage{hyperref,amssymb,amsfonts,amsthm,amsrefs,mathtools,enumerate,verbatim,color,makecell,microtype,bm,nicefrac,stmaryrd}

\newtheorem{Theorem}{Theorem}[section]
\newtheorem{Proposition}[Theorem]{Proposition}
\newtheorem{Lemma}[Theorem]{Lemma}
\newtheorem{Corollary}[Theorem]{Corollary}

\theoremstyle{definition}
\newtheorem{Definition}[Theorem]{Definition}

\newtheorem{Question}[Theorem]{Question}

\newcommand{\rca}{\mathsf{RCA}_0}
\newcommand{\wkl}{\mathsf{WKL}_0}
\newcommand{\aca}{\mathsf{ACA}_0}
\newcommand{\atr}{\mathsf{ATR}_0}

\newcommand{\pica}{\Pi^1_1\mbox{-}\mathsf{CA}_0}

\newcommand{\iso}{\mathsf{I}\Sigma^0_1}

\DeclareMathOperator{\dom}{\mathrm{dom}}

\newcommand{\andd}{\wedge}

\newcommand{\la}{\langle}
\newcommand{\ra}{\rangle}

\newcommand{\imp}{\rightarrow}

\newcommand{\biimp}{\leftrightarrow}

\newcommand{\Nb}{\mathbb{N}}
\newcommand{\Qb}{\mathbb{Q}}
\newcommand{\Rb}{\mathbb{R}}

\newcommand{\cfp}{\mathrm{CFP}}

\usepackage[inline]{enumitem}

\newcommand{\field}[1]{{|{#1}|}}

\newcommand{\sel}{\langle}
\newcommand{\ser}{\rangle}

\newcommand{\vcode}[1]{\stackrel {#1}\rightharpoondown}
\newcommand{\fcode}[1]{\stackrel {#1}\to}
\newcommand{\ballsub}{\subsetplus}
\newcommand{\ballseq}{\subsetpluseq}
\newcommand{\secsub}{\sqsubset}
\newcommand{\secseq}{\sqsubseteq}

\newcommand{\ball}[2]{B_{#2}(#1)}

\newcommand{\spc}[1]{\mathcal{#1}}

\AtBeginDocument{%
   \def\MR#1{}
}

\title{Metric fixed point theory and partial impredicativity}

\author{David Fern\'andez-Duque}
\address{Department of Mathematics WE16\\
Ghent University\\
Ghent\\
Belgium
and
Institute of Computer Science\\
Czech Academy of Sciences\\
Prague\\
Czech Republic}
\email{fernandez@cs.cas.cz}
\urladdr{\url{https://users.ugent.be/~dfernnde/}}

\author{Paul Shafer}
\address{School of Mathematics\\
University of Leeds\\
Leeds\\
UK}
\email{p.e.shafer@leeds.ac.uk}
\urladdr{\url{http://www1.maths.leeds.ac.uk/~matpsh/}}

\author{Henry Towsner}
\address{Department of Mathematics,
University of Pennsylvania\\
Philadelphia\\
USA
}
\email{htowsner@math.upenn.edu}
\urladdr{\url{https://www.sas.upenn.edu/~htowsner/}}

\author{Keita Yokoyama}
\address{Mathematical Institute\\
Tohoku University\\
Sendai\\
Japan
}
\email{keita.yokoyama.c2@tohoku.ac.jp}
\urladdr{\url{http://www.math.tohoku.ac.jp/english/people/yokoyama-e.html}}


\begin{document}

\begin{abstract}
We show that the Priess-Crampe \& Ribenboim fixed point theorem is provable in $\rca$.  Furthermore, we show that Caristi's fixed point theorem for both Baire and Borel functions is equivalent to the transfinite leftmost path principle, which falls strictly between $\atr$ and $\pica$.  We also exhibit several weakenings of Caristi's theorem that are equivalent to $\wkl$ and to $\aca$.
\end{abstract}

\maketitle

\section{Introduction}

Metric fixed point theorems state that, under certain conditions, a
function $f\colon {\spc X}\to {\spc X}$ from a metric space to itself has a fixed
point, i.e.~there is an $x_* \in {\spc X}$ such that $f(x_*) = x_*$.
Such theorems have many applications in geometry, partial differential equations, etc., where  the function $f$ is typically continuous.
 
One fixed point theorem that does {\em not} require the continuity of $f$ is Caristi's fixed point theorem~\cite{Caristi}.
Instead, the function $f$ is `controlled' by a non-negative lower semi-continuous function.
Specifically, it applies to what we will call `Caristi systems'.
As is typical, we will notationally identify a metric space $(\spc X,d)$ with $\spc X$.
Recall that $ V\colon \spc X\to\Rb$ is lower semi-continuous if whenever $x_n\to x$, it follows that $V(x) \leq \displaystyle\liminf_{n\to\infty}V(x_n )$. 

\begin{Definition}
A {\em Caristi system} is a tuple $(\spc X,f,V)$, where $\spc X$ is a complete metric space, $f\colon \spc X\to \spc X$ is arbitrary, $V\colon \spc X\to [0,\infty)$ is lower semi-continuous, and for all $x\in \spc X$, $d(x,f(x))\leq V(x)-V(f(x))$.
\end{Definition}

\begin{Theorem}[Caristi~\cite{Caristi}]\label{TheoWeakCaristi}
Every Caristi system $(\spc X,f)$ has a fixed point; i.e.~there is $x_* \in \spc X$ such that $f(x_*)=x_*$.
\end{Theorem}

Caristi's theorem has various applications and generalizations in metric fixed point theory~\cites{CaristiApps,Kirk2003}.
We think of $V\colon \spc X\to [0,\infty)$ as a `potential', which diminishes after applying $f$; intuitively, after enough applications of $f$, no more potential is lost and we have reached a fixed point.
In Caristi's original proof, $f$ is iterated transfinitely, but this can be avoided by using Ekeland's variational principle.

\begin{Theorem}[Ekeland~\cite{ekeland1974}]\label{TheoEkeland}
Let $\spc X$ be a complete metric space and $V\colon \spc X\to [0,\infty)$ be lower semi-continuous.
Then there is an $x_*\in \spc X$ such that for all $x\in \spc X$, $ d(x_*,x )\leq V(x_*)-V(x ) $ implies that $x = x_*$.
\end{Theorem}

We call such an $x_*$ a {\em critical point} of $V$.
Theorem \ref{TheoWeakCaristi} can be derived from Theorem \ref{TheoEkeland} by observing that any critical point for $V$ will also be a fixed point of $f$, by the assumption that $d(x_*,f(x_*)) \leq V(x_*) -V(f(x_*))$.
Note however that Ekeland's theorem is equivalent to $\pica$, so the transfinite methods are still hidden `under the hood' \cite{EkelandSelecta}.

Another fixed point theorem that was originally proven via infinitary methods (stated in terms of an explicit invocation of Zorn's lemma) is the Priess-Crampe \& Ribenboim theorem, which deals with spherically complete ultrametric spaces.  This result is partially motivated by logic programming~\cite{PriessCrampe2000} and has recently found applications in cut-elimination for ill-founded proofs~\cite{SavateevS21}.

\begin{Definition}
\begin{sloppypar}
A metric space $\spc X$ is an {\em ultrametric space} if for all $x,y,z\in \spc X$, $d(x,y) \leq \max \{d(x,z),d(z,y)\}$.
$\spc X$ is {\em spherically complete} if whenever $\la \overline B_{\rho_i}(x_i) \ra_{i\in\Nb}$ is a decreasing sequence of closed balls, it follows that $\bigcap _{i \in \Nb} \overline B_{\rho_i}(x_i) \not = \varnothing$.
\end{sloppypar}

A function $f\colon \spc X\to \spc X$ is {\em strictly contracting} if for all $x \not= y \in \spc X$, $d(f(x),f(y)) < d(x,y)$.
\end{Definition}

\begin{Theorem}[Priess-Crampe \& Ribenboim~\cite{priess-crampe2011}]\label{theoPriessCrampe}
Let $\spc X$ be a spherically complete ultrametric space and $f\colon \spc X \to \spc X$ be strictly contracting.
Then, $f$ has a unique fixed point. 
\end{Theorem}

Our goal is to determine the strength of the Priess-Crampe \& Ribenboim theorem and Caristi's theorem in the sense of reverse mathematics. 

For Caristi's theorem, in order to deal with `arbitrary' $f$, we consider the case where $f$ is either Baire or Borel, as these are wide classes that can readily be coded within second-order arithmetic. The Caristi theorem for these classes is strictly between $\atr$ and $\pica$, and indeed equivalent to the theory of the transfinite leftmost path principle introduced by Towsner~\cite{Towsner2013}. This shows that transfinite methods cannot be avoided altogether. The reversal, however, requires a fairly complicated choice of $\mathcal{X}$ and $f$: Caristi's theorem only requires such a strong theory because it covers such complicated functions on fairly general spaces.

If we restrict $\mathcal{X}$ or $f$ to nicer examples, Caristi's theorem becomes easier to prove. The case where both $f$ and $V$ are continuous has also been treated by Peng and Yamazaki~\cite{Peng2017}; this case is interesting, as it can already be viewed as a generalization of the Banach fixed point theorem~\cite{CaristiApps}.
Using our previous work on the reverse mathematics of Ekeland's variational principle~\cite{EkelandSelecta}, we extend this treatment to lower semi-continuous $V$.  As we will see, weakened versions of Caristi's theorem are equivalent to either $\wkl$ when $\spc X$ is compact and both functions are continuous, or $\aca$ when either compactness~\cite{Peng2017} or continuity is dropped (but not both).

Regarding the Priess-Crampe \& Ribenboim theorem, surprisingly it may already be proven in $\rca$.
We show this by exhibiting a new constructive proof.

\section{Subsystems of second-order arithmetic}\label{SecSSOA}

We will work within subsystems of second-order arithmetic as in~\cite{SimpsonSOSOA}.
The language is that of Peano arithmetic enriched with variables for sets of natural numbers which may be quantified over.  We use ${\Delta}^0_0$ to denote the set of all formulas, possibly with set parameters, where no second-order quantifiers appear and all first-order quantifiers are bounded, and as usual define the classes $\Sigma^e_n$ and $\Pi^e_n$ where $n$ is the number of alternating first-order (for $e=0$) or second-order (for $e=1$) quantifiers (see e.g.~\cite{SimpsonSOSOA} for details).

We use the notation $\sel x_0,\ldots,x_n\ser$ to denote sequences of natural numbers encoded in a standard way.
As usual, sets of pairs may be used to represent binary relations and functions on the natural numbers, and for a binary relation $R$, $\field R$ denotes the union of the domain and codomain of $R$.%
\footnote{Within $\rca$, $\field R$ may not exist as a set but one can always find an isomorphic relation $R'$ such that $\field{R'}$ exists as a set.}
The set of all finite sequences of natural numbers is denoted $\Nb^{<\Nb}$. For $\sigma,\tau\in\Nb^{<\Nb}$ we write $\sigma\secseq \tau $ if $\sigma$ is an initial segment of $\tau$, $\sigma\secsub \tau $ if $\sigma$ is a proper initial segment of $\tau$, and set $\mathop\downarrow \sigma=\{\tau \in \Nb^{<\Nb} : \tau \secseq \sigma\}$.
If $\Lambda \colon \Nb \to \Nb$ and $n\in \Nb$, write $\Lambda\upharpoonright n$ for the finite sequence $\sel\Lambda(i)\ser_{i<n}$; note that the sequence is empty when $n=0$. For a set $X\subseteq \Nb$ we write $X\upharpoonright n$ for the finite sequence $\lambda_{X}\upharpoonright n$ where $\lambda_{X}$ is the characteristic function of $X$. We extend the use of $\secsub$ by defining $\sigma \secsub \Lambda$ whenever $\sigma = \Lambda\upharpoonright n$ for some $n$.
Concatenation of sequences is denoted by $\frown$.
If $X,Y \subseteq\Nb$ then $X\oplus Y$ is $\{2n : n \in X\} \cup \{2n+1 : n \in Y\}$, the Turing join of $X$ and $Y$.

We will represent trees as subsets of $\Nb^{<\Nb}$ which are downward closed under $\secseq$. Binary trees are then those trees which are subsets of $\{0,1\}^{<\Nb}$.
An infinite sequence $\Lambda$ is a \emph{path through $T$} if for every $n$, $\Lambda\upharpoonright n\in T$, and $[T]$ denotes the class of paths through $T$.

If $\prec$ is a well-order on a subset of $\Nb$, then ${\rm ot}(\prec)$ denotes the order-type of $\prec$.
If moreover $X\subseteq \Nb^2$ is a set of pairs, and $\alpha\in\field\prec$, we write  $X_\alpha$ for $\{x\in \Nb:\sel x,\alpha\ser \in X \}$ and $X_{ \prec \alpha }$ for the set of all $\sel n, \beta \ser \in X$ with $\beta \prec \alpha $, with $X_{ \preccurlyeq \alpha }$ being defined analogously.
Transfinite recursion along $\prec$ may be defined in second-order arithmetic, as follows.

\begin{Definition}[\cite{SimpsonSOSOA}*{Chapter V}]\label{defH}
  Let $\theta(x,Y,\vec z,\vec Z)$ be any formula.
We define $H_\theta(\prec,Y,\vec z,\vec Z)$ to be the formula which says that, for each $\alpha \in \field \prec $, $ Y_\alpha =\{x: \theta(x,Y_{\prec \alpha},\vec z,\vec Z)\}$ and for $\alpha \not\in \field \prec $, $Y_\alpha =\varnothing$.
  We may write $H_\theta(\prec,Y)$ when $\vec z,\vec Z$ are clear from context.
\end{Definition}

We may choose $\theta$ so that for any parameters $Y,\vec z,\vec Z$, the set $\{x: \theta(x,Y,\vec z,\vec Z)\}$ is a universal computably enumerable set relative to $Y,\vec Z$---that is, so that for any $Y,\vec z,\vec Z$, any set computably enumerable relative to $Y,\vec Z$ is equal to some slice $\{x : \theta(\langle e,x\rangle,Y,\vec z,\vec Z)\}$. In this case the choice of $\theta$ only matters up to details of coding, so we omit it.
\begin{Definition}  
For a fixed $\theta_* $ such that for any parameters $Y,\vec z,\vec Z$, the set $\{x: \theta_*(x,Y,\vec z,\vec Z)\}$ is a universal computably enumerable set relative to $Y,\vec Z$, we write $H$ instead of $H_{\theta_*}$.
\end{Definition}

This suffices to define the `Big Five' theories of reverse mathematics: $\rca$ includes basic axioms of arithmetic together with induction for $\Sigma^0_1$-definable predicates and comprehension for $\Delta^0_1$-definable predicates; $\wkl$ extends $\rca$ with the formalized weak K\"{o}nig's lemma; and $\aca$ includes comprehension for arithmetical formulas.
Then $\atr$ ensures that, whenever $\prec$ is a well-order, there is a unique $Y$ so that $H(\prec,Y,\vec z,\vec Z)$ holds, and finally, $\pica$ is axiomatized with comprehension for $\Pi^1_1$ formulas.
We have mentioned these theories in strictly increasing order of strength, but there is a less known theory between $\atr$ and $\pica$ due to the third author~\cite{Towsner2013}.

\begin{Definition}\label{defSigmaPrec}
When $\prec$ is a linear order, we say $\prec$ is a \emph{successor} if it has a maximal element $x$, and we define its \emph{predecessor} ${\prec^-}= {{\prec}\upharpoonright{(\field \prec \setminus\{x\}) }}$.
  When $\prec$ is a well-order, we say $W$ is \emph{$\Sigma_\prec^Z$} if either
  \begin{enumerate*}[label=(\alph*)]
  \item $\prec$ is a successor and $W$ is computably enumerable in the unique $\bar Y:=Y\oplus Z$ so that $H(\prec^-,Y,Z)$ holds, or
  \item $\prec$ is not a successor and $W$ is computable in the $\bar Y:=Y\oplus Z$ so that $H(\prec,Y,Z)$ holds.
  \end{enumerate*}%
\footnote{Note that there is an irregularity at the lowest level  in our notation, namely, $W$ is $\Sigma_{\emptyset}^{Z}$ means that $W$ is $\Delta_{1}^{Z}$ and not $\Sigma_{0}^{Z}$ in the sense that $W$ is defined by a bounded formula from $Z$.}
\end{Definition}

$\mathsf{TLPP}_0$ is defined to be $\rca$ together with the ${\rm TLPP}$ (transfinite leftmost path principle) axiom.
When $\Lambda$ and $\Gamma$ are infinite sequences, we write $\Gamma<\Lambda$ if there is some $n$ such that $\Gamma\upharpoonright n=\Lambda\upharpoonright n$ and $\Gamma(n)<\Lambda(n)$.
Then $\rm TLPP$ is the formalization of the following statement:
 {\em whenever $T \subseteq \Nb^{<\Nb}$ is a tree with an infinite path and $\prec$ is a well-order, there is a path $\Lambda$ through $T$ such that there is no path $\Gamma$ through $T$ which is $\Sigma^{\Lambda\oplus T}_\prec$ and $\Gamma < \Lambda$.}
We call $\Lambda$ a {\em relativized $\prec$-leftmost path for $T$,} or just {\em relativized leftmost path} when $\prec$, $T$ are clear from context.
For a detailed treatment of these and other subsystems of second-order arithmetic, see~\cites{SimpsonSOSOA,Towsner2013}.
The following two characterizations of $\mathsf{TLPP}_0$ will be useful; the proof is simply a relativization of the one in~\cite{MR1428011}.

\begin{Lemma}[$\rca$]\label{thm:tlpp_equivalent}
The following are equivalent:
\begin{enumerate}
\item\label{itTLPPOne} $\mathsf{TLPP}_0$.
\item\label{itTLPPTwo} Whenever $\sel T_n\ser_{n\in\Nb}$ is a sequence of trees and $\prec$ is a well-order, there are a pair of sets $Z_0,Z_1$ so that $n\in Z_0$ if and only if there is a path through $T_n$ which is $\Sigma^{Z_0\oplus Z_1\oplus \sel T_n \ser_{n\in\Nb}}_\prec$.
\item\label{itTLPPThree} For any parameters $\vec Z$, any well-order $\prec$, and any $\Sigma^{0}_{2}$ formula
 $\phi$, there are sets $X_0,X_1$ so that $x\in X_0$ iff for every $Y$ which is $\Sigma_{\prec}^{X_0\oplus X_1\oplus \vec Z}$, $\phi(x,Y,\vec Z)$ holds.
\end{enumerate}
\end{Lemma}
\begin{proof}
To show \ref{itTLPPOne} implies \ref{itTLPPTwo}, let $\sel T_n\ser_{n \in \Nb}$ be given.  We define a single tree $T$ intertwining these trees: we first define $T'_n=\{\sel 0\ser^\frown \sigma: \sigma\in T_n\}\cup\{\sigma:  \forall i<|\sigma| \ \sigma(i)=1\}$ (that is, we add a single infinite path to the right of all paths in $T_n$) and then take $T=\{\sigma: \forall\sel i,n\ser <|\sigma| \ \sel \sigma(\sel 0,n\ser ),\ldots,\sigma(\sel i,n\ser )\ser\in T'_n\}$.  By $\mathsf{TLPP}_0$, let $\Lambda$ be a relativized leftmost path for $T'$ and let $Z=\{n: \Lambda(\sel 0,n\ser )=0\}$.

If $n\in Z$ then there is a path through $T_n$ computable from $\Lambda$.  If $n\not\in Z$ and $\Gamma$ is a path through $T_n$ which is $\Sigma_\prec^{Z\oplus\Lambda \oplus\sel T_n \ser_{n \in \Nb}}$, then we can modify $\Lambda$ by setting $\Lambda'(\sel 0,n\ser )=0$, $\Lambda'(\sel i+1,n\ser )=\Gamma(i)$, and $\Lambda'(\sel i,m\ser )=\Lambda(\sel i,m\ser)$ for $m\neq n$.  But $\Lambda'<\Lambda$ and $\Lambda'$ is $\Sigma_\prec^{\Lambda\oplus \sel T_n\ser_{n \in \Nb} }$, contradicting the choice of $\Lambda$.

Next we show that \ref{itTLPPTwo} implies \ref{itTLPPThree}.  
By the normal form theorem, write $\phi(x,Y,\vec Z)\equiv\neg \forall w \exists z\ \psi(w,z,x,Y\upharpoonright z,\vec Z)$.  
For each $x$, let $T_x$ be the tree defined by $\sigma\in T_{x}$ iff for all $i\le j<|\sigma|$, $\sigma(i)\secseq\sigma(j)\in 2^{\Nb}$, $i\le|\sigma(i)|$ and $\exists z\le |\sigma(i)|\ \psi(i,z,x,\sigma(i)\upharpoonright z,\vec Z)$ holds.

We extend $\sel T_x\ser_{x \in \Nb}$ to encode $\vec Z$ and find sets $S_0,S_1$ so that $x\in S_0$ iff there is a path through $T_x$ which is $\Sigma^{S_0\oplus S_1\oplus\sel T_x\ser_{x \in \Nb}\oplus\vec Z}_\prec$.  The complement $\overline{S_0}$ of $S_0$ is the desired set.  If $x\not\in \overline{S_0}$, then there is a path $\Lambda$ through $T_x$ which is $\Sigma^{S_0\oplus S_1\oplus\sel T_x\ser_{x \in \Nb}\oplus \vec Z}_\prec$. Since $\sel T_x\ser_{x \in \Nb}$ is computable from $\vec Z$, this path is $\Sigma^{\overline{S_0}\oplus S_1\oplus\vec Z}_\prec$ and $Y=\{i:\Lambda(i)(i)=1\}$ witnesses $\forall w\exists z\ \psi(w,z,x,Y\upharpoonright z,\vec Z)$.  On the other hand, if there is a $Y$ so that $\forall w\exists z\psi(w,z,x,Y\upharpoonright z,\vec Z)$ and $Y$ is $\Sigma^{\overline{S_0}\oplus S_1\oplus\vec Z}_\prec$ then a function $\Lambda$ defined by $\Lambda(i)=Y\upharpoonright z_{i}$ where $z_{i}$ is the smallest $z>i$ such that $\psi(i,z,x,Y\upharpoonright z,\vec Z)$ holds is a path through $T_x$ which is $\Sigma^{S_0\oplus S_1\oplus\sel T_x\ser_{x \in \Nb}\oplus\vec Z}_\prec$ and therefore $x\not\in\overline{S_0}$.

To show that \ref{itTLPPThree} implies \ref{itTLPPOne}, let $T$ be a tree with an infinite path $P$, and let $\prec$ be a well-order.  Let $\phi(\sigma,Y,T)$ be a formula which holds iff $Y$ is not a path through $T_\sigma=\{\tau\in T: \sigma \secseq \tau\}$.  By~\ref{itTLPPThree}, let $X_0$ and $X_1$ be such that $\sigma \in X_0$ iff no $Y$ which is $\Sigma_\prec^{X_0 \oplus X_1 \oplus T \oplus P}$ is a path through $T$ extending $\sigma$.  Now recursively compute a path $\Lambda$ from $X_0$ by setting $\Lambda(n)$ to be the least $i$ such that $(\Lambda \upharpoonright n)^\frown \sel i \ser \notin X_0$.  The value $\Lambda(0)$ is defined because $P \restriction 1 \notin X_0$.  If $\Lambda \upharpoonright n$ is defined for some $n \geq 1$, then $\Lambda \upharpoonright n \notin X_0$, so there is a $\Sigma_\prec^{X_0 \oplus X_1 \oplus T \oplus P}$ path $Y$ through $T$ extending $\Lambda \upharpoonright n$.  Then $Y \restriction (n+1) \notin X_0$, which implies that $\Lambda(n)$ is defined and therefore that $\Lambda \restriction (n+1)$ is defined.  Thus $\Lambda$ is a path through $T$.  Now suppose for a contradiction that there is a $\Sigma_\prec^{\Lambda \oplus T}$ path $\Gamma$ through $T$ with $\Gamma < \Lambda$.  Let $n$ be such that $\Gamma \restriction n = \Lambda \restriction n$ and $\Gamma(n) < \Lambda(n)$.  Then $\Gamma \restriction (n+1) \notin X_0$ because $\Gamma$ is $\Sigma_\prec^{\Lambda \oplus T}$ and hence $\Sigma_\prec^{X_0 \oplus X_1 \oplus T \oplus P}$.  This contradicts the definition of $\Lambda(n)$ because $(\Lambda \restriction n)^\frown \sel \Gamma(n) \ser = \Gamma \restriction (n+1) \notin X_0$, but $\Gamma(n) < \Lambda(n)$.
\end{proof}

The unrelativized version of (iii)---the existence of a set $X$ so that $x\in X$ iff for \emph{every $Y$}, $\phi(x,Y,\vec Z)$ holds, is sometimes called an \emph{impredicative} definition (since the set $Y$ might be $X$ itself, or something defined from $X$). The relativized version is called a \emph{partial impredicativity} because we only consider those $Y$'s which are defined from $X$ in a limited way.

\section{Baire and Borel functions}\label{sec-MetricDefs}

Part of the appeal of second-order arithmetic as a foundational system for mathematics is that it suffices to develop a large part of mathematical analysis, particularly when dealing with separable metric spaces.
However, this requires some coding machinery. In this section, we recall this machinery, and establish notation that will be used throughout the text.

\begin{Definition}[$\rca$; see~\cite{SimpsonSOSOA}*{Definition~II.5.1}]
A (code for a) complete separable metric space $\mathcal X=\widehat X$ is defined in $\rca$ to be a nonempty set $X\subseteq \Nb$ together with a sequence of real numbers $d\colon X\times X\to \Rb$ such that $d(a, a) = 0$, $d(a, b) = d(b, a) \geq 0$, and $d(a, b)+d(b, c) \geq d(a, c)$ for all $a, b, c \in X$. A {\em point} of $\widehat X$ is a sequence $x = \la x_i \ra_{i \in \Nb}$ of elements of $X$ such that for all $i\leq j$, $d(x_i , x_j)\leq  2^{-i}$. We write $x\in \widehat X$ to mean that
$x$ is a point of $\widehat X$. We set $d(x,y)=\lim_{n\to\infty}d(x_n,y_n)$, which provably exists in $\rca$.
\end{Definition}

We say that the space $\spc X$ is {\em compact} if there is a sequence of points witnessing that $\spc X$ is totally bounded~\cite{SimpsonSOSOA}*{Definition~III.2.3}.  For our purposes, it suffices to mention that $\rca$ proves that both $[0,1]$ and the Cantor space are compact in this sense.
Here, the Cantor space is $\{0,1\}^\Nb$ with $d(\Lambda,\Lambda') = 2^{-n}$, for the least $n$ such that $\Lambda(n) \neq \Lambda'(n)$ and $d(\Lambda,\Lambda') =0$ when no such $n$ exists.
The {\em Baire space} is defined analogously, but with set of points $\Nb^\Nb$; note that the Baire space is {\em not} compact.
In the definitions below, $\mathbb Q^{>0}$ denotes $\{q \in \mathbb Q: q > 0 \}$.

\begin{Definition}[$\rca$; see~\cite{SimpsonSOSOA}*{Definition~II.5.1}]
Let $\widehat X$ be a complete separable metric space. The {\em (code for the) rational open ball $\ball ar$} is the ordered pair $\la a,r\ra$, with $a\in X$ and $r\in\mathbb Q^{>0}$.  We define $\ball ar \ballsub \ball bq$ if $d(a,b) + r < q$ and $\ball ar \ballseq \ball bq$ if $d(a,b) + r \leq q$. 
\end{Definition}

We remark that $\ball ar\ballseq \ball bq$ implies that $\ball ar\subseteq \ball bq$ (in the usual set-theoretic sense), but not necessarily the converse (consider e.g.~the case where $\spc X$ is a singleton).

One challenge in formalizing Caristi's theorem in the context of second-order arithmetic is that it applies to {\em arbitrary} functions $f$, which would in principle require a third-order quantification.
Instead, we will work with a rather wide class of functions that can still be formalized as second-order objects: Baire and Borel functions.
To keep our presentation unified and to minimize coding concerns, we view all functions on metric spaces as special cases of Baire or Borel functions, which is not how e.g.~continuous or lower semi-continuous functions are typically coded in the literature.
In Appendix~\ref{App} we discuss the relationship between our codes and the more standard ones.

The general theory of Borel sets in reverse mathematics is well established, for instance in~\cites{MR1197207,SimpsonSOSOA}, but Borel \emph{functions} have been less well studied.

\begin{Definition}
  When $\mathcal{X}$ is a complete separable metric space, recall that
  a \emph{Borel code} (that is, a code for a Borel subset of
  $\mathcal{X}$) is a tree of sequences $S$ such that:
  \begin{enumerate*}[label=(\alph*)]
  \item there is no infinite path through $S$,
  \item there is exactly one $n$ so that $\sel n\ser \in S$, and
  \item any leaf $\sigma\in S$ has the form
    $\tau^\frown \sel \ball ar \ser$, where $a\in X$ and $r\in\Qb^{>0}$.
  \end{enumerate*}

  We view such a tree as coding a Borel set.
  In fact, every node in the tree will code a Borel set $U_S(\sigma)$ defined recursively by
  \begin{itemize}
  \item when $\sigma=\upsilon^\frown \sel \ball ar \ser $ is a leaf,
    $U_S(\sigma)=\ball ar$,
  \item when $\sigma=\upsilon^\frown\sel n\ser$ where $n$ is
    odd,
    $U_S(\sigma)=\bigcup  \{U_S(\sigma^\frown\sel k\ser) : { \sigma^\frown\sel k\ser\in
      S} \}$,
  \item when $\sigma=\upsilon^\frown\sel n\ser$ where $n$ is
    even,
    $U_S(\sigma)=\bigcap \{U_S(\sigma^\frown\sel k\ser) : {  \sigma^\frown\sel k\ser\in
      S} \}$,
  \item when $\sigma=\sel\ser$,
    $U_S(\sigma)=U_S(\sel n\ser)$, where $n$ is unique so that
    $\sel n\ser\in S$.
  \end{itemize}
\end{Definition}

In particular, open sets may be represented as unions of open balls, hence via a Borel code.
By abuse of notation, we write $x\in S$ as an abbreviation for $x\in U_S(\sel\ser)$.  Note that the latter is not a formula of second-order arithmetic, however it is a theorem that, in $\atr$, we can identify membership in coded Borel sets (see \cite{SimpsonSOSOA}*{Lemma~V.3.3 and Definition~V.3.4}).

\begin{Lemma}[$\atr$]
If $\sel x_n \ser_{n \in \Nb}$ is a sequence of points in $\mathcal{X}$ and $\sel S_n \ser_{n \in \Nb}$ is a sequence of Borel codes then $\{n: x_n \in U_{S_n}(\sel\ser)\}$ exists.
\end{Lemma}

We will need to discuss Borel functions.  We encode a Borel function as one where the inverse image of each basic open set is given by a Borel code.

\begin{Definition}\label{defBorel}
Let $\spc X$ and $\spc Y$ be metric spaces and $\Upsilon$ be a set coding a sequence of Borel codes $\sel\Upsilon_{\sel a,r\ser}: {\sel a,r\ser \in Y\times \mathbb Q_{>0}}\ser$.
Write $\Upsilon(\ball ar)$ for $U_{\Upsilon_{\sel a,r\ser }}(\sel\ser)$.
 
 We say that $\Upsilon$ is a \emph{(code for) a Borel function} from $\mathcal{X}$ to $\mathcal{Y}$ if for all basic open balls $B,B'$ of $\spc Y$:
  \begin{enumerate}[label=(\roman*)]
  
  \item\label{itBorelOne} if $B  \subseteq B'$ then $\Upsilon(B) \subseteq \Upsilon (B')$,
  \item\label{itBorelTwo} if $B\cap B'=\varnothing$ then $\Upsilon(B)  \cap \Upsilon (B) =\varnothing$,
  
  \end{enumerate} 

  The set $\Upsilon$ codes the partial function $f\colon \spc X\to\spc Y$, where $x\in \dom(f)$ if there 
 is a (unique) element $f(x)=y\in \spc Y$ such that for every ball $B$ containing $y$, $x\in \Upsilon (B)$.

Recall that the {\em Kleene-Brouwer order} of a tree $T$, denoted ${\text{\sc{kb}}}(T)$, is a linearization of $T$ which, provably in $\aca$, remains well founded when $T$ is (see e.g.~\cite{SimpsonSOSOA}).
When $f$ is a Borel function coded by $\Upsilon$, the \emph{complexity} of $f$, written $\|f\|$, is ${\text{\sc{kb}}}(\{\sel n\ser^\frown\sigma: \sigma\in \Upsilon_n\})$.
\end{Definition}

  In particular, continuous functions may be coded as Borel functions, where the preimage of every open ball is open.

\begin{Lemma}[$\atr$]\label{thm:borel_bounds_complexity}
  Let $\Upsilon$ be a code for a Borel function from $\mathcal{X}$ to $\mathcal{Y}$.  Then, for any $x\in\mathcal{X}$, the image of $x$ under $f$ exists and is $\Sigma_{\|f\|}^{x\oplus \Upsilon}$.
\end{Lemma}
\begin{proof}
The evaluation function
of the set of balls $B_r(a)$ such that $x\in f^{-1}(B_r(a))$ is $\Sigma_{\|f\|}^{x\oplus \Upsilon}$.  Therefore we may define an approximation by taking $y_i$ to be the least $a$ such that $x\in f^{-1}(B_{2^{-i-1}}(a))$, and then the sequence $\sel y_i \ser_{i \in \Nb}$ converges to $f(x)$.
\end{proof}

Rather than referring to codes, we will usually talk about the function $f$: saying that $f$ is coded is simply saying that, for each $B_r(a)$, the set $f^{-1}(B_r(a))$ is given by a Borel code, and these codes are presented uniformly in $a,r$.

\begin{Definition}
  A \emph{(code for a) partial Baire function} is a well-founded tree $\Xi$ such that each leaf is labeled by a (code for a) partial continuous function and each non-leaf has an extension for each $i\in\mathbb{N}$.

For any $x$ and any $\sigma\in \Xi$, the value $f_{\Xi,\sigma}(x)$, if it exists, is defined recursively by:
\begin{itemize}
\item if $\sigma$ is a leaf labeled by $f$ then $f_{\Xi,\sigma}(x)=f(x)$,
\item if $\sigma$ is not a leaf and there is any $n$ so that $f_{\Xi,\sigma^\frown\sel n\ser}(x)$ does not exist then $f_{\Xi,\sigma}(x)$ does not exist,
\item if $\sigma$ is not a leaf, $f_{\Xi,\sigma^\frown\sel n\ser}(x)$ exists for all $n$, then $f_{\Xi,\sigma}(x)=\lim_{n\rightarrow\infty}f_{\Xi,\sigma^\frown\sel n\ser}(x)$ if this exists, and does not exist otherwise.
\end{itemize}

We write $f_\Xi(x)$ for $f_{\Xi,\sel\ser}(x)$ and call $f$ a {\em Baire function} if it is given by a code for a partial Baire function such that, for every $x$, $f(x)$ exists.
\end{Definition}

It is provable in $\atr$ that every Baire function is Borel, and the converse is true on zero-dimensional spaces, including the Cantor space and Baire space; see Appendix \ref{App} for details.

Potentials (i.e., non-negative lower semi-continuous functions) can be represented as increasing limits of continuous functions and hence as Baire class $1$ functions.
For this, we use the following lemma (not formalized in second-order arithmetic).

\begin{Lemma}\label{lemmLSCrepresentInformal}
If $\spc X$ is any metric space and $V$ is a potential on $\spc X$, then there exists a pointwise increasing sequence of continuous functions $V_n\colon \spc X\to [0,\infty)$ such that $V = \lim_{n\to\infty} V_n$.\end{Lemma}

\begin{proof}[Proof sketch]
It is easy to check that an increasing limit of continuous functions is lsc.
For the converse, given $\alpha \geq 0$, define the {\em $\alpha$-envelope of $V$,} denoted $V_{(\alpha)}$, by $V_{(\alpha)}(x ) = \inf_{y\in\spc X} \big (V(y) + \alpha d(x,y)\big )$.
Then, it is not hard to check that $(V_{(n)})_{n \in\mathbb N} $ is an increasing sequence of continuous functions converging pointwise to $V$.
\end{proof}

Thus in second-order arithmetic we define a {\em potential} to be a Baire class $1$ function which is an limit of pointwise-monotone continuous functions $V_n$ (i.e.~$V_n(x) \leq V_m(x)$ if $n \leq m$).
If $V_{n}$ is the $n$-envelope of $V$, we say that $V$ is {\em enveloped.} Note that $V_{(\alpha)}$ is well defined even when $\alpha=0$, in which case we obtain $V_{(0)} = \inf V$, and thus enveloped potentials have an infimum.
As we will see, there is much more information that can be extracted from continuous envelopes.

This coding machinery will suffice to formalize Caristi's theorem, but first we turn our attention to Priess-Crampe, which involves only continuous functions.

\section{An elementary proof of the Priess-Crampe \& Ribenboim theorem}

Recall that Theorem \ref{theoPriessCrampe} states that if $\spc X$ is a spherically complete ultrametric space and $f\colon \spc X \to \spc X$ is strictly contracting, then $f$ has a unique fixed point.
This theorem may be stated in second-order arithmetic via the coding machinery described above.  For spherical completeness, let $\overline B_\rho(x)$ denote the class of $y \in \spc X$ such that $d(x,y) \leq \rho$ for given $x \in \spc X$ and $\rho > 0$.  As $\spc X$ is an ultrametric space, we write $\overline B_\rho(x) \ballseq \overline B_\delta(y)$ to denote that $\max\{d(x,y), \rho\} \leq \delta$ and observe that $\overline B_\rho(x) \ballseq \overline B_\delta(y)$ implies that $\overline B_\rho(x) \subseteq \overline B_\delta(y)$.  Then, $\spc X$ is spherically complete if whenever $\la \overline B_{\rho_i}(x_i)\ra_{i \in \Nb}$ is a sequence of closed balls such that $\overline B_{\rho_i}(x_i) \supsetpluseq \overline B_{\rho_{i+1}}(x_{i+1})$ for all $i$, there is an $x \in \bigcap_{i \in \Nb}\overline B_{\rho_i}(x_i)$.  We now obtain the following.

\begin{Theorem}
The Priess-Crampe \& Ribenboim theorem
is provable in $\rca$.
\end{Theorem}

\proof
Uniqueness follows easily from the assumption that $f$ is strictly contracting, so we focus on existence.

Let $\spc X = \widehat X$, and let $\la a_i \ra_{i\in\Nb}$ enumerate $X$ with each $a_i$ occurring infinitely often.
Define $\rho(x) = d(x,f(x))$.
We claim that there exists a sequence $\la b_i \ra_{i\in\Nb}$ such that $\rho(b_i) \geq \rho(b_{i+1})$ for all $i$ and $b_i \to \inf \rho$ as $i\to \infty$ (in the sense that for all $x\in \spc X$ and $\varepsilon>0$, there is an $i$ with $\rho(b_i) < \rho(x) +\varepsilon$).
Construct the sequence $(b_i)_{i\in\Nb}$ as follows.
Note that $\rho(x) <\rho(y)$ is a ${ \Sigma}^0_1 $ statement, which we may represent as $\exists z \phi(x,y,z)$.
Let $b_0 = a_0$, and recursively define $b_{i+1} = a_{i+1}$ if there is a $z < i$ witnessing that $\rho(a_{i+1}) < \rho(b_i)$ (in the sense that $\phi(a_{i+1},b_i,z)$ holds), and otherwise define $b_{i+1} = b_i$.
It is not hard to check that the sequence $\la b_i \ra_{i\in\Nb}$ satisfies the required properties.

Let $\rho_{i} = \rho(b_i)$ for each $i$.  We show that $\overline B_{\rho_i}(b_i) \supsetpluseq \overline B_{\rho_{i+1}}(b_{i+1})$ for all $i$.  To see this, observe that
\begin{align*}
d(b_i, b_{i+1}) &\leq \max\{d(b_i, f(b_i)), d(f(b_i), f(b_{i+1})), d(b_{i+1}, f(b_{i+1}))\}\\
&= \max\{\rho_i, d(f(b_i), f(b_{i+1})), \rho_{i+1}\} = \rho_i.
\end{align*}
\begin{sloppypar}
The last equality holds because $f$ is strictly contracting and therefore $d(f(b_i), f(b_{i+1})) < d(b_i, b_{i+1})$.  Thus it must be that $\max\{\rho_i, d(f(b_i), f(b_{i+1})), \rho_{i+1}\} = \max\{\rho_i, \rho_{i+1}\} = \rho_i$.  It follows that $\max\{d(b_i, b_{i+1}), \rho_{i+1}\} \leq \rho_i$, so $\overline B_{\rho_i}(b_i) \supsetpluseq \overline B_{\rho_{i+1}}(b_{i+1})$.
\end{sloppypar}

By spherical completeness, there is an $x_* \in \bigcap_{i \in \Nb} \overline B_{\rho_i}(b_i)$.  Then $\rho(x_*) \leq \rho_i$ for all $i$ because
\begin{align*}
d(x_*,f(x_*)) \leq \max\{d(x_*,b_i), d(b_i,f(b_i)), d(f(b_i), f(x_*))\} = \max\{d(x_*,b_i), \rho_i\} = \rho_i,
\end{align*}
since $x \in \overline B_{\rho_i}(b_i)$ and $f$ is strictly contracting.
If $\rho(x_*)>0$ then $d(x_*,f(x_*))>d(f(x_*),f^2(x_*)) = \rho(f(x_*))$, which contradicts the fact that $\lim_{i\to \infty}\rho(b_i) = \inf \rho$.
We conclude that $d(x_*,f(x_*)) = 0$.
\endproof

\section{The Caristi theorem in the Big Five}\label{SecFormCaristi}

In this section we study weakenings of Caristi's theorem provable in $\wkl$ and $\aca$.
Caristi's theorem follows directly from Ekeland's variational principle, so we first recall the main results from~\cite{EkelandSelecta} regarding the reverse mathematics of the latter.
In particular, Ekeland's variational principle is derivable in $\pica$, thus establishing an upper bound for Caristi's theorem.  Ekeland's variational principle also has natural weakenings derivable in $\wkl$ and in $\aca$.

\begin{Definition}\label{DefFormalCritical}
Given definable classes $\mathfrak X$ of coded metric spaces and $\mathfrak V$ of coded potentials, the {\em (formalized) Ekeland variational principle for $\spc X\in\mathfrak X$ and $V\in\mathfrak V$} is the statement that, if $\spc X\in\mathfrak X$ is a coded separable complete metric space and $V \in \mathfrak V$ is a coded potential, then there exists $x_\ast \in \spc X$ such that for all $x\in \spc X$, if $d(x_*,x) \leq V(x_*)-V(x)$, then $x=x_*$.

When not mentioned, we assume that $\mathfrak X$ is the class of all coded complete separable metric spaces and $\mathfrak V$ is the class of all coded potentials.
\end{Definition}

The following is proven in~\cite{EkelandSelecta}, and in fact all items reverse.

\begin{Theorem}\label{thm-CritPtsExist}
The Ekeland variational principle holds:
\begin{enumerate}[label=(\roman*)]

\item \textup{(}$\wkl$\textup{)} For compact $\spc X$ and continuous $V$.

\item\label{CritPtsExistTwo} \textup{(}$\aca$\textup{)} For compact $\spc X$ or continuous $V$.

\item\label{CritPtsExistFour} \textup{(}$\pica$\textup{)} For the class of all metric spaces with an arbitrary potential.

\end{enumerate}
\end{Theorem}

All of these results reverse, although we won't be needing this.
We remark that~\cite{EkelandSelecta} uses a different presentation of potentials (i.e.~lower semi-continuous functions), but the two are equivalent over $\rca$ (see Appendix~\ref{App}).
The discontinuous cases are established by reducing to continuous cases via envelopes, given the following~\cite{EkelandSelecta}*{Lemma 8.1}.

\begin{Lemma}[$\rca$]\label{lemmEnvelopeCrit}
Let $V$ be any potential and  $\alpha>1$ be such that $V_{(\alpha)}$ exists.
Then, any critical point $x_*$ of $V_{(\alpha)}$ is also a critical point of $V$, and $V_{(\alpha)} (x_*) = V  (x_*)$. 
\end{Lemma}

We are now ready to state our formalization of the Caristi fixed point theorem and prove it (and its weakenings) in standard systems of second-order arithmetic.

\begin{Definition}\label{DefFormalCaristi}
A {\em Caristi system} is a tuple $(\spc X, f,V)$, where $\spc X$ is a coded complete separable metric space, $f \colon \spc X \to \spc X$ a coded Baire or Borel function, and $V \colon \spc{X} \to [0, \infty)$ a coded potential.

Given definable classes $\mathfrak X$ of complete, separable metric spaces and $\mathfrak F,\mathfrak V$ of functions such that elements of $\mathfrak F$ are Borel or Baire functions of the form $f\colon \spc X\to \spc X$ and elements of $\mathfrak V$ are potentials of the form $V \colon \spc X \to [0, \infty)$, the {\em (formalized) Caristi fixed point theorem ($\cfp$) for $\spc X\in\mathfrak X$, $f\in\mathfrak F$, and $V\in\mathfrak V$} is the statement that, if $(\spc X,f,V)$ is a Caristi system whose elements belong to the aforementioned classes, then there exists $x_\ast \in \spc X$ such that $x _\ast =f(x_\ast)$.

When not mentioned, we assume that $\mathfrak X$ is the class of all coded separable complete metric spaces, $\mathfrak F$ is the class of all Baire or Borel functions, and $\mathfrak V$ is the class of all coded potentials.
\end{Definition}

The following is a corollary of Theorem \ref{thm-CritPtsExist}, using the fact that the Ekeland variational principle implies Caristi's theorem.

\begin{Proposition}\label{Prop_holds}
The $\cfp$ holds whenever:
\begin{enumerate}[label=(\roman*)]

\item  \textup{(}$\wkl$\textup{)} $\spc X$ is compact and $V$ is continuous.

\item\label{itCarACA}  \textup{(}$\aca$\textup{)} $\spc X$ is compact, $f$ is continuous, or $V$ is continuous.

\item  \textup{(}$\pica$\textup{)} Always.

\end{enumerate}
\end{Proposition}

\begin{proof}
Most items follow by using Theorem~\ref{thm-CritPtsExist} and the fact that every critical point of $V$ is a fixed point of $f$.
The exception is \ref{itCarACA} for continuous $f$, but the proof of Peng and Yamazaki~\cite{Peng2017}, itself a version of the classic proof of the Banach fixed point theorem, works in this context.

Let $\langle \spc X,f,V\rangle$ be a Caristi system where $f$ is continuous and $x_0\in \spc X$.
Define a sequence $x_n$ given recursively by $x_{n+1} = f(x) $; this sequence is readily available in $\aca$ using the continuity of $f$.
Similarly, the sequence given by $v_n = V(x_n)$ exists, as $\aca$ can compute suprema uniformly.
Since $d(x_n,f(x_n)) \leq V(x_n) - V(f(x_n)) =  V(x_n) - V(x_{n+1}) $, we must have that $V(x_{n+1}) \leq V(x_n)$, hence $\langle v_n\rangle_{n=0}^\infty$ is a decreasing sequence of real numbers and thus Cauchy.
The inequality $d(x_n,x_m) \leq V(x_n) - V(x_m) $ for $m>n$ implies that $\langle x_n\rangle_{x=0}^\infty$ is also Cauchy, hence it has a limit, say $x_\infty$.
By the continuity of $f$ and the definition of $x_n$ we see that $x_\infty = \lim_{n\to\infty}x_n =
 \lim_{n\to\infty}f(x_n) = f(\lim_{n\to\infty}x_n) = f(x_\infty) $, so indeed, $x_\infty$ is a fixed point of $f$.
\end{proof}

In the rest of this work we will show that all of these items reverse, except for (iii). (Note that the function in (iii) is required, by our definition of a Caristi system, to be Baire or Borel.)

\begin{Proposition}\label{prop_cantorWKL}
The $\cfp$ for $\spc X = 2^\Nb$ and both $f$ and $V$ continuous implies $\wkl$ over $\rca$.
\end{Proposition}

\begin{proof}
We prove the contrapositive over $\rca$.  Assume that $\wkl$ fails, and let $T \subseteq 2^{<\Nb}$ be an infinite binary tree with no infinite path.  We define continuous $f \colon 2^\Nb \to 2^\Nb$ and $V \colon 2^\Nb \to [0,3]$ so that $(2^\Nb, f, V)$ is a Caristi system with no fixed point.  Our $V$ is the potential from the proof that Ekeland's variational principle for continuous potentials on $2^\Nb$ implies $\wkl$ of~\cite{EkelandSelecta}*{Proposition~9.1}.

As in the proof of~\cite{EkelandSelecta}*{Proposition~9.1}, let $T^\circ = \{\sigma \in T : \sigma^\frown 0, \sigma^\frown 1\notin T\}$ be the set of leaves of $T$.  The set $T^\circ$ is infinite because $T$ is infinite but has no path.  For each $\sigma \in T^\circ$, let
\begin{align*}
A_\sigma = \{i < |\sigma| : \neg(\exists \tau \in T)(|\tau| = |\sigma|+1 \wedge \tau \sqsupseteq (\sigma \upharpoonright i)^\frown(1-\sigma(i))\}.
\end{align*}
For each $\sigma \in 2^{<\Nb}$, let $\tilde{\sigma} \in 2^{<\Nb}$ be the sequence of length $2|\sigma|$ where $\tilde{\sigma}(2i) = 0$ and $\tilde{\sigma}(2i+1) = \sigma(i)$ for all $i < |\sigma|$.  Define $\tilde{T} = \{\tilde{\sigma} : \sigma \in T\}$, $\tilde{T}^\circ = \{\tilde{\sigma} : \sigma \in T^\circ\}$, and
\begin{align*}
S = \{\tau\in 2^{<\Nb} : (\forall \sigma \in T)(|\sigma| \leq |\tau| \imp \tau \not\sqsubseteq \tilde{\sigma}) \wedge (\exists \sigma \in T)(|\sigma| \leq |\tau| \wedge \tau \upharpoonright (|\tau|-1) \sqsubset \tilde{\sigma})\}.
\end{align*}
The set $S$ consists of the shortest binary sequences that move away from $\tilde{T}$ before reaching an element of $\tilde{T}^\circ$.  The elements of $\tilde{T}^\circ \cup S$ are pairwise incomparable, and the fact that $T$ has no infinite path implies that for every $x \in 2^\Nb$, there is a $\sigma \in \tilde{T}^\circ \cup S$ with $\sigma \sqsubseteq x$.  Define the continuous function $V \colon 2^\Nb \to [0,3]$ by
\begin{align*}
V(x) =
\begin{cases}
 2 - \sum_{i \in A_\sigma} 2^{-2i} & \text{if $x \sqsupseteq \tilde{\sigma}$ for a $\sigma\in T^\circ$}\\
 3 & \text{if $x \sqsupseteq \tau$ for a $\tau \in S$}.
\end{cases}
\end{align*}
One may obtain a code for $V$ via~\cite{EkelandSelecta}*{Lemmas~3.8~and~3.9}.

Before defining $f$, observe that for any $\sigma \in T^\circ$, there is a $\tau \in T$ with $|\tau| = |\sigma|+1$ because $T$ is infinite.  Such a $\tau$ does not extend $\sigma$ because $\sigma$ is a leaf, so there is an $i < |\sigma|$ such that $\tau \sqsupseteq (\sigma \upharpoonright i)^\frown(1-\sigma(i))$.  Thus there is an $i < |\sigma|$ with $i \notin A_\sigma$.  Given $\sigma \in T^\circ$, let $i_\sigma$ be greatest such that $i_\sigma < |\sigma|$ and $i_\sigma \notin A_\sigma$, and let $\sigma_+$ be the first element of $T^\circ$ in lexicographic order with $|\sigma_+| > |\sigma|$ and $\sigma_+ \sqsupseteq (\sigma \upharpoonright i_\sigma)^\frown(1-\sigma(i_\sigma))$.  Then $i_\sigma \in A_{\sigma_+}$ because any $\tau \in T$ with $\tau \sqsupseteq (\sigma_+ \upharpoonright i_\sigma)^\frown(1-\sigma_+(i_\sigma))$ satisfies $\tau \sqsupseteq \sigma \upharpoonright (i_\sigma + 1)$ and therefore also satisfies $|\tau| \leq |\sigma| < |\sigma_+|$ by the maximality of $i_\sigma$.  Furthermore, if $i < i_\sigma$ and $i \in A_\sigma$, then $i \in A_{\sigma_+}$ as well because $\sigma_+ \upharpoonright i_\sigma = \sigma \upharpoonright i_\sigma$ and $|\sigma| < |\sigma_+|$.

Fix the first $\eta \in T^\circ$.  Define the continuous function $f \colon 2^\Nb \to 2^\Nb$ following $V$ by
\begin{align*}
f(x) = 
\begin{cases}
{\tilde{\sigma}_+}{}^\frown 0^\Nb & \text{if $x \sqsupseteq \tilde{\sigma}$ for a $\sigma \in T^\circ$}\\
{\tilde{\eta}}^\frown 0^\Nb & \text{if $x \sqsupseteq \tau$ for a $\tau \in S$}.
\end{cases}
\end{align*}
Again, one obtains a code for $f$ via~\cite{EkelandSelecta}*{Lemmas~3.8~and~3.9}.
The function $f$ has no fixed point.
However, $(2^\Nb, f, V)$ is a Caristi system.
Let $x \in 2^\Nb$.
If $x \sqsupseteq \tilde{\sigma}$ for a $\sigma \in T^\circ$, then $d(x, f(x)) = 2^{-2i_\sigma-1}$ and
\begin{align*}
V(x) - V(f(x)) &= \sum_{i \in A_{\sigma_+}}2^{-2i} - \sum_{i \in A_\sigma}2^{-2i} \geq 2^{-2i_\sigma} - \sum_{\substack{i \in A_\sigma \\ i > i_\sigma}}2^{-2i}\\
&\geq 2^{-2i_\sigma} - 2^{-2i_\sigma - 1} = 2^{-2i_\sigma - 1} = d(x, f(x)),
\end{align*}
where the first inequality is because $i_\sigma \in A_{\sigma_+} \setminus A_\sigma$ and because $i \in A_\sigma \imp i \in A_{\sigma_+}$ for $i < i_\sigma$.  If instead $x \sqsupseteq \tau$ for a $\tau \in S$, then
\begin{align*}
V(x) - V(f(x)) = 1 + \sum_{i \in A_\eta}2^{-2i} \geq d(x, f(x)).
\end{align*}
Therefore $(2^\Nb, f, V)$ is a Caristi system with no fixed point, which completes the proof.  
\end{proof}

Peng and Yamazaki~\cite{Peng2017} showed that $\aca$ is equivalent to the $\cfp$ for continuous $f$ and $V$.
We sharpen this result by showing that we may assume that $\spc X$ is the Baire space.

\begin{Proposition}\label{prop-CaristiContinuousReversal}
The $\cfp$ for $\spc X = \Nb^\Nb$ and both $f$ and $V$ continuous implies $\aca$ over $\rca$.
\end{Proposition}

\begin{proof}
We prove the contrapositive over $\rca$.  Assume that $\aca$ fails, and let $h \colon \Nb \to \Nb$ be an injection whose range does not exist as a set.  Let $T \subseteq \Nb^{<\Nb}$ be the tree constructed from $h$ as in the proof that K\"{o}nig's lemma implies $\aca$ of~\cite{SimpsonSOSOA}*{Theorem~III.7.2}.  This $T$ is the set of all $\sigma \in \Nb^{<\Nb}$ such that $(\forall m < |\sigma|)(\forall n < |\sigma|)(h(m) = n \biimp \sigma(n) = m+1)$ and $(\forall n < |\sigma|)(\sigma(n) > 0 \imp h(\sigma(n)-1) = n)$.  The tree $T$ does not have an infinite path because the range of $h$ does not exist as a set.  We define continuous $f \colon \Nb^\Nb \to \Nb^\Nb$ and $V \colon \Nb^\Nb \to [0,3]$ so that $(\Nb^\Nb, f, V)$ is a Caristi system with no fixed point.

To every $\sigma \in T$, assign a $\sigma_+ \in T$ with $|\sigma_+| > |\sigma|$ and $\{n < |\sigma_+| : \sigma_+(n) > 0\} \supseteq \{n < |\sigma| : \sigma(n) > 0\}$ as follows.  Given $\sigma \in T$, let $k = \max\{|\sigma|, \max\{\sigma(n) : n < |\sigma|\}\}$.  Let $X$ be the finite set  $X = \{h(m) : m \leq k\}$, and let $\sigma_+$ be the sequence of length $k+1$ where for $n \leq k$  
\begin{align*}
\sigma_+(n) =
\begin{cases}
m+1 & \text{if $n \in X$ and $h(m) = n$}\\
0 & \text{if $n \notin X$}.
\end{cases}
\end{align*}
Then $\sigma_+ \in T$.  Furthermore, if $n < |\sigma|$ and $\sigma(n) > 0$, then $\sigma(n) \leq k$, so $h(\sigma(n)-1) = n \in X$, so $\sigma_+(n) = \sigma(n) > 0$.

Define continuous $f \colon \Nb^\Nb \to \Nb^\Nb$ and $V \colon \Nb^\Nb \to [0,3]$ as follows, using~\cite{EkelandSelecta}*{Lemmas~3.8~and~3.9} to obtain the codes.
Given $x \in \Nb$, let $\sigma \sqsubseteq x$ be the longest initial segment with $\sigma \in T$, which exists because $x$ is not a path through $T$.  Let
\begin{align*}
V(x) &= 1 + 2^{-|\sigma| + 1} - \sum_{\substack{n < |\sigma| \\ \sigma(n) > 0}}2^{-n} &
f(x) &= {\sigma_+}^\frown 0^\Nb.
\end{align*}
Notice that $f(x) \neq x$ because $\sigma_+ \in T$ but $x \upharpoonright |\sigma_+| \notin T$ as $|\sigma_+| > |\sigma|$.  So $f$ has no fixed point.  To see that $d(x, f(x)) \leq V(x) - V(f(x))$, first observe that the longest initial segment $\tau \sqsubseteq f(x)$ with $\tau \in T$ has the form $\tau = {\sigma_+}^\frown 0^\ell$ for some $\ell$ because $\sigma_+ \in T$.  Therefore
\begin{align*}
V(f(x)) = 1 + 2^{-|\tau|+1} - \sum_{\substack{n < |\tau| \\ \tau(n) > 0}}2^{-n} \leq 1 + 2^{-|\sigma_+| + 1} - \sum_{\substack{n < |\sigma_+| \\ \sigma_+(n) > 0}}2^{-n}
\end{align*}
because $|\sigma_+| \leq |\tau|$ and $\{n < |\sigma_+| : \sigma_+(n) > 0\} = \{n < |\tau| : \tau(n) > 0\}$.  Thus
\begin{align*}
V(x) - V(f(x)) &\geq 2^{-|\sigma| + 1} - 2^{-|\sigma_+| + 1} + \sum_{\substack{n < |\sigma_+| \\ \sigma_+(n) > 0}}2^{-n} - \sum_{\substack{n < |\sigma| \\ \sigma(n) > 0}}2^{-n}\\
&\geq 2^{-|\sigma|} + \sum_{\substack{n < |\sigma_+| \\ \sigma_+(n) > 0}}2^{-n} - \sum_{\substack{n < |\sigma| \\ \sigma(n) > 0}}2^{-n},
\end{align*}
where the second inequality is because $|\sigma| < |\sigma_+|$.  If $\sigma \sqsubseteq \sigma_+$, then $d(x, f(x)) = 2^{-|\sigma|}$, so
\begin{align*}
d(x, f(x)) = 2^{-|\sigma|} \leq 2^{-|\sigma|} + \sum_{\substack{n < |\sigma_+| \\ \sigma_+(n) > 0}}2^{-n} - \sum_{\substack{n < |\sigma| \\ \sigma(n) > 0}}2^{-n} \leq V(x) - V(f(x)),
\end{align*}
where the first inequality holds because $\{n < |\sigma_+| : \sigma_+(n) > 0\} \supseteq \{n < |\sigma| : \sigma(n) > 0\}$.  If $\sigma \not\sqsubseteq \sigma_+$, then let $j < |\sigma|$ be least with $\sigma(j) \neq \sigma_+(j)$.  It must be that $\sigma(j) = 0$ and $\sigma_+(j) > 0$ because if $\sigma(j) > 0$, then $\sigma_+(j) > 0$ as well, in which case $\sigma(j) = \sigma_+(j)$ as both $\sigma$ and $\sigma_+$ are in $T$.  Therefore
\begin{align*}
2^{-j} \leq \sum_{\substack{n < |\sigma_+| \\ \sigma_+(n) > 0}}2^{-n} - \sum_{\substack{n < |\sigma| \\ \sigma(n) > 0}}2^{-n}
\end{align*}
because $\sigma_+(j) > 0$, $\sigma(j) = 0$, and $\{n < |\sigma_+| : \sigma_+(n) > 0\} \supseteq \{n < |\sigma| : \sigma(n) > 0\}$.  Thus
\begin{align*}
d(x, f(x)) = 2^{-j} \leq 2^{-|\sigma|} + \sum_{\substack{n < |\sigma_+| \\ \sigma_+(n) > 0}}2^{-n} - \sum_{\substack{n < |\sigma| \\ \sigma(n) > 0}}2^{-n} \leq V(x) - V(f(x)).
\end{align*}
So $d(x,f(x)) \leq V(x) - V(f(x))$ in both cases.  Therefore $(\Nb^\Nb, f, V)$ is a Caristi system with no fixed point, which completes the proof.
\end{proof}

\begin{Proposition}\label{propBaire1toACA}
The $\cfp$ for $[0,1]$ with Baire class $1$ $f$ implies $\aca$ over $\rca$.
\end{Proposition}

\begin{proof}
We work in $\rca$ and prove the contrapositive.  Over $\rca$, $\aca$ is equivalent to the monotone convergence theorem (see~\cite{SimpsonSOSOA}*{Theorem~III.2.2}).  In fact, by inspecting the proof of~\cite{SimpsonSOSOA}*{Theorem~III.2.2}, $\aca$ is equivalent to the statement ``every strictly increasing sequence of rationals in $[0,1]$ has a supremum.''  Thus we let $\sel c_n \ser_{n \in \Nb}$ be a strictly increasing sequence of rationals in $[0,1]$ with no supremum,
and we define a Caristi system $\langle [0,1], f,V \rangle$ with $f$ Baire class $1$, but with no fixed points.

First we define $V$.
Let $V_n \colon [0,1] \to [0,2]$ be a piecewise linear function such that $V_n(x) = 2$ for $x\leq c_n$, $V_n(x) = x$ for $x\geq c_{n+1}$, and $V_n $ descends linearly on $[c_n,c_{n+1}]$.
Let $V = \lim_{n\to\infty} V_{n}$.
Clearly the functions $V_n$ are continuous and increasing on $n$, so $V$ is a potential and it is easy to see that
\begin{align*}
V(x) =
\begin{cases}
2 & \text{if $\exists n(x < c_n)$}\\
x & \text{otherwise}.
\end{cases}
\end{align*}

Now we define $f$, show that $([0,1], f,V)$ is a Caristi system, and show that $f$ has no fixed points.  Define a sequence $\sel f_n \ser_{n \in \Nb}$ of piecewise-linear functions $f_n \colon [0,1] \imp [0,1]$ as follows.  For each $n$, first find the sequence $\sel q^n_i \ser_{i \leq n+2}$, where
\begin{itemize}
\item $q^n_0 = 1$;
\item for each $i \leq n$, $q^n_{i+1}$ is the rational in $(c_n, q^n_i)$ with the least code;
\item $q^n_{n+2} = c_n$.
\end{itemize}
Now define $f_n$ so that
\begin{itemize}
\item for $i \leq n$, $f_n$ is linear on $[q^n_{i+1}, q_i^n]$ with $f(q^n_{i+1}) = q^n_{i+2}$ and $f(q^n_i) = q^n_{i+1}$;

\item $f_n$ is linear on $[q^n_{n+2}, q^n_{n+1}]$ with $f(q^n_{n+2}) = 1$ and $f(q^n_{n+1}) = q^n_{n+2}$;

\item $f_n$ is constantly $1$ on $[0, q^n_{n+2}]$.
\end{itemize}
We show, for every $x \in [0,1]$, that $f(x) = \lim_{n \imp \infty}f_n(x)$ exists, that $d(x,f(x)) \leq V(x) - V(f(x))$, and that $f(x) \neq x$.  Let $x \in [0,1]$.  First suppose that there is an $n_0$ such that $x < c_{n_0}$.  In this case, $f_n(x) = 1$ for all $n \geq n_0$, so $\lim_{n \imp \infty}f_n(x) = 1$.  Thus $f(x) = 1 \neq x$.  Moreover, $d(x, f(x)) \leq 1 = 2 - 1 = V(x) - V(f(x))$.

Now suppose that $\forall n (c_n \leq x)$.  As $x$ is not the supremum of $\sel c_n \ser_{n \in \Nb}$, there is a rational $v$ such that $v \leq x$ and $\forall n (c_n \leq v$).  Thus by $\iso$ in the guise of the $\Pi^0_1$ least number principle, there is such a $v$ whose code is least.  Similarly, there is a rational $u$ whose code is least such that $u < v$ and $\forall n (c_n \leq u$).  Let $Q$ be the finite set of rationals $\geq u$ whose codes are at most the code of $u$.  Let $\sel q_i \ser_{i \leq m}$ be the longest sequence of elements from $Q \cup \{1\}$ where $q_0 = 1$ and, for each $i < m$, $q_{i+1}$ is the element of $[0, q_i) \cap Q$ with the least code.  Observe that $q_m = u$  and $q_{m-1} = v$.  Now, by the choice of $u$, let $n_0 \geq m$ be large enough so that $p \leq c_{n_0}$ for every rational $p \in [0,u)$ whose code is less than the code of $u$.  Then $n \geq n_0$ implies that $(\forall i \leq m)(q^n_i = q_i)$.  Thus, as $x \in [q_{m-1}, q_{m-2}]$, $n \geq n_0$ implies that $f_n(x) = \frac{q_{m-1} - q_m}{q_{m-2} - q_{m-1}}(x - q_{m-1}) + q_m$.  So $f(x) = \lim_{n \imp \infty} f_n(x) = \frac{q_{m-1} - q_m}{q_{m-2} - q_{m-1}}(x - q_{m-1}) + q_m$.  We thus have that $q_m \leq f(x) < x$.  Hence $\forall n (c_n \leq q_m \leq f(x))$ because $q_m = u$.  Therefore $V(x) = x$ and $V(f(x)) = f(x)$, so $d(x, f(x)) = x - f(x) = V(x) - V(f(x))$.

We have now defined a potential $V \colon [0,1] \imp [0, \infty)$ and a Baire class $1$ function $f \colon [0,1] \imp [0,1]$ such that $([0,1], f,V)$ is a Caristi system but such that $f$ has no fixed points.  This completes the proof.
\end{proof}

Let us put together some of our results so far:

\begin{Theorem}[$\rca$]
The following are equivalent:

\begin{enumerate}[label=(\alph*)]

\item \label{ItWKL} $\wkl$;

\item the $\cfp$ for continuous $V$ and compact $\spc X$;

\item the $\cfp$ for continuous $f,V$ on the Cantor space.

\end{enumerate}

\end{Theorem}

\begin{proof}
  That (a) implies (b) is Proposition \ref{Prop_holds}, (c) is a special case of (b), and (c) implies (a) is Proposition \ref{prop_cantorWKL}.
\end{proof}

\begin{Theorem}[$\rca$]\label{TheoCritEquivACA}
The following are equivalent:

\begin{enumerate}[label=(\alph*)]

\item \label{ItACA} $\aca$;

\item the $\cfp$ for either $f$ or $V$ continuous;

\item the $\cfp$ for compact $\spc X$;

\item the $\cfp$ for Baire class $1$ $f$ on $[0,1]$;

\item the $\cfp$ for continuous $f,V$ on the Baire space.

\end{enumerate}

\end{Theorem}
\begin{proof}
  Proposition \ref{Prop_holds} shows that (a) implies both (b) and (c). Item (d) is a case of (c) and (e) is a case of (b). That  (d) implies (a) is Proposition \ref{propBaire1toACA} and  (e) implies (a) is Proposition \ref{prop-CaristiContinuousReversal}.
\end{proof}

We remark that the strength of the $\cfp$ for compact $\spc X$, continuous $f$, and {\em arbitrary} $V$ is left open, although by the above, it must fall between $\wkl$ and $\aca$.

\section{Caristi's theorem and the TLPP}\label{SecCFPvsTLPP}
In this section we will show that the unrestricted $\cfp$ is equivalent to $\mathsf{TLPP}_0$; recall that we defined the latter in Section \ref{SecSSOA}.

We first prove Caristi's theorem in ${\sf TLPP}_0$.
We derive this from a `relativized' version of Ekeland's variational principle provable in ${\sf TLPP}_0$.
The idea is that for a potential $V$, we can find a potential $V'$ which is the same as $V$ on hyperarithmetically definable points (with suitable parameters) and so that $V'$ has a critical point.
Below, we let $\spc X^Y_\prec$ be the set of all points of $\spc X$ that are $\Sigma^Y_\prec$.

\begin{Definition}
Let $\spc X$ be a complete separable metric space and $V$ be a potential on $\spc X$.
Let $Z\subseteq \Nb$.
A {\em $\Sigma_\prec^Z$ relativization of $V$} is a potential $V' = \lim_{n\to\infty}V'_n$ whose Baire code is arithmetical on some set $W$ and such that for every $x\in \spc X^{W\oplus Z}_\prec$ and $n>0$,
\begin{equation}\label{eqInf}
V'_{n}(x) = \inf_{y \in \spc X_{\prec}^{W\oplus Z} } ( V(y) + n   d(x,y) ).
\end{equation}
We call $W$ the {\em parameter} of $V'$.
\end{Definition}

Roughly speaking, $V'$ represents a ``smoothed'' version of $V$ in which discontinuities at points of high complexity are forgotten.

\begin{Lemma}[${\sf TLPP}_0$]\label{lemmRelEk}
For every potential $V$ and every set $Z$, there is a $\Sigma^{  Z}_\prec$ relativization of $V$.
\end{Lemma}

\begin{proof}
Fix $Z$ and $V$.
Note that for any set $W$, the function $V'_{n-1}(x) $  defined by \eqref{eqInf} is uniformly continuous (when defined), because for $x,x',y \in \spc X$, we have that
\[|(V(y) + n d(x,y))- (V(y) + n d(x',y))|\leq n d(x,x'),\]
hence $|V'_n(x) - V'(x')| \leq  n d(x,x') $.
It is moreover evident that $V'_n(x)\leq V'_m(x)$ whenever $n\leq m$, given that $nd(x,y)\leq m d(x,y)$ for all $x,y$.
It remains to show that $\lim_{n\to\infty} V'_n $ has a Baire code as a sequence of continuous functions.

Let $\mathcal{X}$ be a complete separable metric space, let $V \colon \spc X \to [0, \infty)$ be a potential, and let $Z$ be any set.
The property $c < V(y) +n d(a,y) $ is $\Sigma^0_1$, hence using $\mathsf{TLPP}_0$ in the form of Lemma~\ref{thm:tlpp_equivalent} item~\ref{itTLPPThree}, we obtain sets $\Delta$ and $Y$ so that for $x\in X$, $(a, c,n) \in \Delta$ if and only if for all  $y \in  \spc X_{\prec}^{\Delta \oplus Y \oplus Z} $, $c< V(y) + n d(a,y)$.
We set $W:= \Delta\oplus Y$ and write $\spc X' $ for $\spc X_{\prec}^{W \oplus Z} $.

From $\Delta$, we may arithmetically define a code $\Upsilon_n$ for $V'_n$.
Let us represent open balls in $\Rb$ as intervals $(c,d)$.
Then, enumerate $\ball ar $ into $ \Upsilon_n((c,d)) $ if there are $c',d'$ such that $c<c'<d'<d$ and for all $b\in \ball ar\cap X$, $V'_n(b) \in (c',d')$.
We check that $\Upsilon_n$ thus defined satisfies Definition~\ref{defBorel} and codes the desired function $V'_n$.

For Item~\ref{itBorelOne}, it is immediate that if $(c_0,d_0)\subseteq (c_1,d_1)$ and $\ball ar $ is in $ \Upsilon_n((c_0,d_0)) $ via $(c',d')$, then the same sub-interval witnesses that $\ball ar $ is in $ \Upsilon_n((c_1,d_1)) $.
For Item~\ref{itBorelTwo}, if $ \Upsilon_n((c_0,d_0)) \cap   \Upsilon_n((c_1,d_1)) \neq \varnothing $, since these sets are open and $X$ is dense, we may find $a\in \Upsilon_n((c_0,d_0)) \cap   \Upsilon_n((c_1,d_1)) \cap X $, which implies that $V'_n(a) \in (c_0,d_0) \cap   (c_1,d_1)$, so the two intervals must intersect.

It remains to check that $\Upsilon_n$ indeed codes the desired function $V'_n$ of \eqref{eqInf}.
Let $x\in\spc X$ and write $x=\lim_{i\to\infty} x_i$, with $x_i\in X$.
By uniform continuity, the value $V'_n(x)$ as given by \eqref{eqInf} is $\lim_{i\to\infty} V'_n(x_i)$, which exists in $\aca$ as a real number.
Now, suppose that $ (c,d) $ is an interval containing $V'_n(x)$, and let $c',d'$ be such that $c<c'<V'_n(x)<d'<d$.
Since $V'_n$ is continuous, for small enough $\delta$, we see that $V'_n[\ball x\delta] \subseteq (c',d') $.
Let $a\in X$ be such that $d(a,x)<\nicefrac\delta 2$, so that $x\in \ball a{\nicefrac \delta 2} \subseteq \ball x\delta$.
Then $V'_n[\ball a{\nicefrac\delta 2}] \subseteq (c',d') $, so $V'_n[\ball a{\nicefrac\delta 2}]$ is enumerated in $\Upsilon_n ((c,d))$ and $x\in \Upsilon_n ((c,d))$.
Since $c,d$ were arbitrary, we see that $x$ is in the domain of the coded function by $\Upsilon_n$ and $V'_n(x)$ is indeed the value assigned to $x$.
\end{proof}

When looking at points of not-too-high complexity, we want $V'$ to look essentially the same as $V$.
The next lemma makes this precise.

\begin{Lemma}\label{lemmPrimeSame}
Let $V$ be a potential on $\spc X$, and let $V' = \lim_{n\to\infty} V'_n$ be a $\Sigma^Z_\prec$ relativization of $V$.
Then, $V(x)=V'(x)$ for every $x\in \spc X'$.
\end{Lemma}

\begin{proof}
Let $x\in \spc X'$.
From the definition of $V'_n$, we see that we may instantiate $y$ as $x$ and obtain $V'_n(x) \leq V(x) + n d(x,x) = V(x)$, so $V'(x)\leq V(x)$.

To show that $V(x) \leq V'(x)$, we show that $V(x) \leq V'(x) + \varepsilon$ for all rational $\varepsilon > 0$.
It suffices to show that for all $\varepsilon > 0$ there is an $n$ such that $V(x) \leq V_n'(x) + \varepsilon$.

Since $V$ is lsc, let $\delta$ be so that $d(x,y) <\delta$ implies $V(y) \geq V(x) - \varepsilon$.
Choose $n$ so that $n\delta  > V(x)  $.
Then, $V(x) \leq V(y) + nd(x,y) + \varepsilon$ for every $y \in \mathcal{X}'$: if $d(x,y)<\delta $ this is because $V(x) \leq V(y) +   \varepsilon$ by our choice of $\delta$, otherwise $V(x) \leq n\delta \leq nd(x,y) $ by our choice of $n$.
We conclude that $V(x) \leq V_n'(y) + \varepsilon$, as needed.
\end{proof}

As mentioned, $V'$ is meant to be a version of $V$ where discontinuities of high complexity are removed.
The following makes this precise.

\begin{Lemma}\label{lemmEpsilonSquare}
Let $V\colon\spc X\to[0,\infty)$ be any potential and $V'$ a $\Sigma^Z_{\prec}$-relativization of $V$.
If $x\in \spc X$ and $\varepsilon>0$, then there is $y\in \spc X'$ such that $d(x,y)<\varepsilon$ and $V'(y) < V'(x) + \varepsilon$.
\end{Lemma}

\begin{proof}
Assume toward a contradiction that $x\in \spc X$ and $\varepsilon>0$ are such that $V'(y) \geq V'(x) + \varepsilon$ whenever $y\in \spc X'$ is such that $d(x,y)<\varepsilon$.

Choose $n > \nicefrac{V'(x)}\varepsilon +1 $ and let $y\in \spc X'$ be arbitrary. 
If $d(x,y)<\varepsilon$,
\[V'(y) + n d(x,y) \geq V'(y) \geq V'(x) +\varepsilon ,\]
while if $d(x,y)\geq \varepsilon$, we see by our choice of $n$ that
\[   V'(y)+n d(x,y) \geq   n \varepsilon >(\nicefrac{V'(x)}\varepsilon +1)\varepsilon =  V'(x) +\varepsilon.\]
But $V(y) = V'(y)$ by Lemma~\ref{lemmPrimeSame}, so we conclude that $V(y)+n d(x,y) \geq V'(x) + \varepsilon$ for all $y\in\spc X'$.
We obtain $V'_n(x)>V'(x)+\varepsilon$ by \eqref{eqInf}, contradicting $V'_n(x)\leq V'(x)$.
\end{proof}

With this, we can show that $V'$ is indeed enveloped.

\begin{Lemma}[$\rca$]\label{lemmRelHon}
Let $V$ be a potential on $\spc X$, and let $V' = \lim_{n\to\infty} V'_n$ be a $\Sigma^Z_\prec$ relativization of $V$.
Then, $V'_n(x)=V'_{(n)}(x)$ for all $n$.
\end{Lemma}

\begin{proof}
Fix $x\in \spc X$.
First we show for $y\in \spc X$ arbitrary that $V_n '(x) \leq   V'(y) + nd(x,y) $.
Let $\varepsilon>0$.
Using Lemma \ref{lemmEpsilonSquare}, let $y'\in \spc X'$ be such that $d(y,y')<\nicefrac\varepsilon{2(n+1)}$ and $V ' (y')  < V '(y)+ \nicefrac \varepsilon 2$.
Then,
\begin{align*}
V'_n (x) & \leq   V (y') +nd(x,y') & \text{by definition of $V'_n$,}\\
& = V'(y') +nd(x,y') & \text{since $y'\in\spc X'$,}\\
& < V'(y) + n d(x,y) + \varepsilon & \text{by our choice of $y'$.} 
\end{align*}
Since $\varepsilon$ was arbitrary, $V'_n(x)  \leq V'(y) + nd(x,y) $.

Finally we check that if $\varepsilon>0$, then there is $y \in \spc X$ such that $V'_n(x) +\varepsilon>V'(y) +n d(x,y) $.
By definition, $ V'_n(x) = \inf_{y\in \spc X'} (V(y) + nd(x,y)) $, so we can choose $y \in \spc X'$ such that $V'_n(x) +  \varepsilon  > V(y) +n d(x,y)$.
Since $y\in \spc X'$, $V'(y) = V(y)$, so that
\[ V'_n(x) + \varepsilon  >  V'(y) + nd(x,y) ,\]
as required.
\end{proof}

Since $\aca$ proves that all enveloped potentials have a critical point by Theorem \ref{thm-CritPtsExist} and Lemma~\ref{lemmEnvelopeCrit}, we obtain the following.

\begin{Corollary}[$\aca$]\label{corRelCrit}
Every $\Sigma^Z_\prec$ relativization with parameter $W$ has a critical point that is arithmetical in $W \oplus Z$ (hence $\Sigma_\prec^{W \oplus Z}$ if $\field\prec$ is infinite).
\end{Corollary}

With this, we may prove Caristi's theorem for Baire or Borel functions.

\begin{Proposition}[${\sf TLPP}_0$]\label{propTLPPtoCar}
Caristi's fixed point theorem holds for arbitrary lower semi-continuous potentials $V$ and arbitrary Baire or Borel functions $f$.
\end{Proposition}

\begin{proof}
Since ${\sf TLPP}_0$ extends $\atr$, we may appeal to Lemma \ref{lemBaireisBorel} to see that every Baire function is Borel, so we may assume that $f$ is Borel.

Let $\mathcal{X}$ be a complete separable metric space, let $\Phi\subseteq\mathbb{N}\times X\times \mathbb{Q}^{>0}\times\mathbb{Q}$ code a lower semi-continuous potential on $\mathcal{X}$, and let $f$ be a class-${\rm ot}(\prec)$ Borel function coded by $\Upsilon$; it is convenient to assume that ${\rm ot}(\prec)$ is infinite.
By Lemma \ref{lemmRelEk}, $V$ has a $\Sigma_\prec^{ \Upsilon}$ relativization $V'$ with parameter $W$, and by Corollary~\ref{corRelCrit}, $V'$ has a critical point $x_*$ that is arithmetical in ${  W \oplus \Upsilon}$.
Since ${\sf TLPP}_0$ extends $\atr$, we may use Lemma \ref{thm:borel_bounds_complexity} to see that $f(x_*)$ is $\Sigma_\prec^{   W \oplus \Upsilon}$.
We have $d(x_*,f(x_*))\leq V(x_*)-V(f(x_*))$, and since $V(f(x_*))=V'(f(x_*))$, $V(x_*)=V'(x_*)$, and $x_*$ is a critical point of $V'$, this implies that $x_*=f(x_*)$, as required.
\end{proof}

Our next goal is to prove that the Caristi fixed point theorem for Baire/Borel functions is equivalent to ${\sf TLPP}_0 $ over $\rca$.
As a first step, we show that $\atr$ is already provable at a rather low stage in the Baire hierarchy.

\begin{Theorem}\label{thm:caristi_for_baire_1}
Over $\rca$, Caristi's theorem for Baire class-$1$ $f$ implies $\atr$.
\end{Theorem}
\begin{proof}
Since Caristi's theorem for continuous functions implies $\aca$, we can work over $\aca$.
Recall that the formula $H$ describes the transfinite recursion for $\theta_*(n,Y)$ which defines a universal computably enumerable set relative to $Y$. Here, we may safely use a parameter $Y\in\Nb^{\Nb}$ instead of a member of $2^{\Nb}$.
By the normal form theorem,
we write $\theta_{*}(n,Y)\equiv\exists m\,\theta_{0}(m,n,Y\upharpoonright m)$.

In what follows, we will emulate the transfinite recursion on $\Nb^{\Nb}$.
Let $\prec$ be a well-order on $\Nb$.
We will construct a sequence of trees $\sel T_{\alpha}: \alpha\in\Nb\ser$ by $\Delta^{0}_{1}$-transfinite recursion which is implied from $\aca$ (see~\cites{Freund2020,DFSW}).
For a given $\alpha\in\Nb$ and $\sel T_{\beta}\subseteq \Nb^{<\Nb}: \beta\prec \alpha \ser$, we let $T^{\alpha}=\bigoplus_{\beta\prec\alpha}T_{\beta}:=\{\sigma\in \Nb^{<\Nb}:$ for any $\beta <|\sigma|$,  $\sigma_{\beta}\in T_{\beta}$ if $\beta \prec \alpha$ and $\sigma_{\beta}$ is a sequence of $0$'s otherwise$\}$,
where $\sigma_{i}\in \Nb^{<\Nb}$ is $\sigma_{i}(x)=\sigma(\sel x,i\ser )$ if $\sel x,i\ser<|\sigma|$. (Here, the domain $\Nb$ may be identified with $\Nb\times\Nb$ by the pairing function.
Without loss of generality, we may assume that $\sel x,i\ser \geq i$ in the sense of $\Nb$.)

Next we construct $T_{\alpha}$ from $T^{\alpha }$.
Here, the range $\Nb$ may be identified with $\{0,1\} \times \Nb^{<\Nb}$, and for a given $n=\sel s,\tau\ser\in\Nb$, we write $s=(n)_{\circ}$ and $\tau=(n)^{\circ}$. (The idea here is that the second coordinate encodes a path $g\in[T^{\alpha}]$ and the first coordinate encodes the Turing jump of $g$.)
Then we define $T_{\alpha}$ as follows: $\sigma\in T_{\alpha}$ if
\begin{itemize}
 \item for any $s<|\sigma|$, $s\le |(\sigma(s))^{\circ}|$ and $(\sigma(s))^{\circ}\in T^{\alpha }$,
 \item for any $t<s<|\sigma|$, $(\sigma(t))^{\circ}\secsub(\sigma(s))^{\circ}$,
 \item for any $s<|\sigma|$, $(\sigma(s))_{\circ}=0\to (\forall m\le |(\sigma(|\sigma|-1))^{\circ}|)\neg\theta_{0}(m,s,(\sigma(|\sigma|-1))^{\circ}\upharpoonright m)$, and
 \item for any $s<|\sigma|$, $(\sigma(s))_{\circ}=1\to (\exists m\le |(\sigma(s))^{\circ}|)\theta_{0}(m,s,(\sigma(s))^{\circ}\upharpoonright m)$.
\end{itemize}

Now, let $\sel T_{\alpha }\subseteq \Nb^{<\Nb}: \alpha\in \Nb\ser$ be the result of the transfinite recursion.
If $g\in[T_{\alpha}]$ for some $\alpha \in\Nb$, 
we let $(g)^{\circ}:=\bigcup\{(g(s))^{\circ}:s\in\Nb\}$ and $(g)_{\circ}:=\{g: (g(s))_{\circ}=1\}$, then $(g)^{\circ}\in T^{\alpha}$ and $(g)_{\circ}=\{n:\theta_{*}(n,(g)^{\circ})\}$.
Moreover, let $f_{\beta}:=((g)^{\circ})_{\beta}=\bigcup\{((g(s))^{\circ})_{\beta}:s\in\Nb\}$ for $\beta\prec \alpha$, and $f_{\alpha}=g$.
then, by the definition, we have $f_{\beta}\in[T_{\beta}]$ for any $\beta\preccurlyeq \alpha$.
Now, by arithmetical transfinite induction up to $\alpha$, we may verify that $f_{\beta}$ computes a set $Y^{\beta}$ such that $H_{\theta}(\prec_{\beta},Y^{\beta})$ for any $\beta\preccurlyeq \alpha$, where $\prec_{\gamma}$ is a restriction of $\prec$ to the domain $\{\gamma'\in\Nb: \gamma'\prec\gamma\}$.
Hence, if $\hat f\in [\bigoplus_{\alpha\in\Nb}T_{\alpha}]$, then $(\hat f)_{\alpha}\in [T_{\alpha}]$ for any $\alpha\in\Nb$, and thus $\hat f$ computes a set $Y$ such that $H(\prec,Y)$.

Finally, we construct a Caristi system whose fixed point is a path of  $\bigoplus_{\alpha\in\Nb}T_{\alpha}$.
We define a potential function $V:\mathbb{N}^{\mathbb{N}}\rightarrow[0,\infty)$ by $V(f)=\sum\{2^{-\alpha}: (f)_\alpha\not\in[T_\alpha],\alpha\in\Nb\}$; this defines a total potential on $\mathbb{N}^{\mathbb{N}}$ provably in $\aca$ (\cite{EkelandSelecta}, together with Lemma~\ref{lemmLSCCompare}).
We will construct a Baire class $1$ function $F: \Nb^\Nb \to \Nb^\Nb$ which `descends along' $V$.
For each $s\in\Nb$, we will define a continuous function $F_{s} : \Nb^\Nb \to \Nb^\Nb$ as follows.
For a given $f \in \Nb^\Nb$, let $I_{f,s}:=\{\alpha<s: (f\upharpoonright s)_{\alpha}\notin T_{\alpha},\alpha\in\Nb\}$.
If $I_{f,s}=\varnothing$, then put $F_{s}(f)=f$.
Otherwise, let $\beta_{s}$ be the $\prec$-smallest element of $I_{f,s}$.
Put $(F_{s}(f))_{\gamma}=(f)_{\gamma}$ if $\gamma\neq \beta_{s}$, and $(F_{s}(f))_{\beta_{s}}$ to be the $s$-approximation of a path of $T_{\beta_{s}}$ computed from (the $s$-approximation of)
$f^{\beta_{s}}:=\bigoplus_{\gamma\prec \beta_{s}}f_{\gamma}$.
(Note that $f^{\beta_{s}}\upharpoonright s\in T^{\beta_{s}}$ by the definition.)
More formally, we define $h=(F_{s}(f))_{\beta_{s}}$ as follows:
for $t<s$ define $h(t)$ inductively as $h(t)=(0,f^{\beta_{s}}\upharpoonright\bar t)$ if $(\forall m\le s)\neg\theta_{0}(m,t,f^{\beta_{s}}\upharpoonright m)$ and $\bar t=\max\{t\}\cup\{|(h(t'))^{\circ}|:t'<t\}$ and $h(t)=(1,f^{\beta_{s}}\upharpoonright\bar t)$ if $\bar t$ is the smallest $u\le s$ such that $(\exists m\le  u)\theta_{0}(m,t,f^{\beta_{s}}\upharpoonright m)$ and $u\ge\max\{t\}\cup\{|(h(t'))^{\circ}|:t'<t\}$, and put $h(t)=0$ for $t\ge s$.  It is a routine to check that such $F_{s}$ is continuous.

If $f\in [\bigoplus_{\alpha\in\Nb}T_{\alpha}]$, then $F_{s}(f)=f$ for any $s\in\Nb$.
Otherwise, let $\hat \beta$ be the $\prec$-smallest $\alpha$ such that $(f)_{\alpha}\notin[T_{\alpha}]$.
Then, for large enough $s$, $\beta_{s}=\hat \beta$.
Thus, $g=\lim_{s\to\infty}F_{s}(f)$ exists and $(g)_{\gamma}=(f)_{\gamma}$ for any $\gamma\neq \hat \beta$ and $(g)_{\hat \beta}\in[T_{\hat \beta}]$.
Moreover, we have $d(f,g)\le 2^{-\hat \beta}\le V(f)-V(g)$.
Therefore, $F=\lim_{s\to\infty}F_{s}$ exists, $(\Nb^{\Nb},F,V)$ forms a Caristi system and its fixed point is a path of $\bigoplus_{\alpha\in\Nb}T_{\alpha}$.
\end{proof}

\begin{Question}
  Is Caristi's theorem for Baire class-$1$ $f$ provable in $\atr$?
\end{Question}

\begin{Theorem}[$\rca$]
The following are equivalent:
  \begin{enumerate}
  \item $\mathsf{TLPP}_0$,
  \item $\mathsf{CFP}$ for Baire $f$,
  \item $\mathsf{CFP}$ for Borel $f$, and
  \item $\mathsf{CFP}$ for Borel $f$ on the Baire space.
  \end{enumerate}
\end{Theorem}
\begin{proof}
By Proposition \ref{propTLPPtoCar}, the first item implies the third.
The first moreover implies the second, since $\mathsf{TLPP}_0$ implies $\atr$ and so by Lemma \ref{lemBaireisBorel} shows that every Baire function is a Borel function.
To see that the second implies the fourth, Theorem \ref{thm:caristi_for_baire_1} shows that it implies $\atr$, and therefore by Lemma \ref{thm:borel_is_baire}, it implies that every Borel function from $\mathbb{N}^{\mathbb{N}}$ to itself is a Baire function, so the fourth follows.

To complete the loop, it remains to show that the fourth implies the first.
By Theorems~\ref{TheoCritEquivACA} and \ref{thm:caristi_for_baire_1}, we may argue within $\atr$.
Using Lemma \ref{thm:tlpp_equivalent}, we will show the equivalent form: given a sequence of trees $\sel T_n \ser_{n \in \Nb}$ and a well-order $\prec$, we will obtain a pair of sets $Z_0,Z_1$ so that $n\in Z_0$ iff there is a path through $T_n$ which is $\Sigma_\prec^{Z_0\oplus Z_1\oplus\sel T_n \ser_{n \in \Nb}}$.
To show this, we define $V:\Nb^\Nb \to [0, \infty)$ by $V(f)=\sum\{2^{-i}: (f)_i\not\in[T_i]\}$ as in the proof of Theorem \ref{thm:caristi_for_baire_1}.

We wish to define a Borel function $F:\mathbb{N}^{\mathbb{N}}\rightarrow\mathbb{N}^{\mathbb{N}}$ by having $F(f)$ look for a $g$ which is a counterexample to $f$ being a critical point of $V$ and which is not too much more complicated than $f$.  Let $\prec'$ be the successor of $\prec$.  Then pick a uniform enumeration $\psi_0, \psi_1, \psi_2, \ldots$ of the partial computable functions and, given $f$, we let $Y_f$ be the set such that $H(\prec',Y_f,f\oplus\sel T_n \ser_{n \in \Nb})$.  Then for each $e$ such that $\psi_e^{Y_f}$ is total, we can let $g_e=\psi_e^{Y_f}$.  We wish to choose $F(f)$ to be $g_e$ where $e$ is least such that $0<d(f,g_e)\leq V(f)-V(g_e)$ if there is such an $e$, and to be $f$ otherwise.

We need to check that $F$ is Borel.  It suffices to show that, for any finite sequence $\sigma \in \Nb^{<\Nb}$, the set of $f$ such that there is $e$ such that $\psi^{Y_f}_e$ exists and satisfies $\sigma\secsub g_e$ and $0<d(f,g_e)\leq V(f)-V(g_e)$ is Borel.  

The main step is translating the construction of $Y_f$ into a statement about Borel sets.  By recursion on $\alpha \in\field {\prec'}$ we argue that the set of $f$ such that $m\in (Y_f)_\alpha$ is Borel: when $\alpha$ is minimal, $m\in (Y_f)_\alpha$ iff $\theta_{*}(m,\varnothing,f\oplus\sel T_n \ser_{n \in \Nb})$, and since $\theta_{*}$ is $\Sigma^{0}_{1}$, the set of such $f$ is open, namely a union of those initial segments that witness this.

Suppose that, for all $m$, the set of $f$ such that $m\in (Y_f)_{\prec\alpha} $ is Borel.  Then $m\in(Y_f)_\alpha$ if and only if $\theta_{*}(m,(Y_f)_{\prec\alpha},f\oplus\sel T_n\ser_{n\in\Nb})$.  This is a union of sets which are, by the recursion, Borel sets intersected with open sets, and is therefore also a Borel set.
In particular, for each $e$, the set of $f$ such that $\psi_e^{Y_f}$ is total is a Borel set. 

The set of $f$ such that $V(f)<q$ is easily seen to be Borel.
Since, for any $i$ and $j$, the set of $f$ such that $\psi_e^{Y_f}(i)=j$ is Borel, also for any rational $q$, the set of $f$ such that $0<d(f,g_e)<q$ is Borel, as is the set of $f$ such that $ V(g_e)<q$.  Then the set of $f$ such that $0<d(f,g_e)\leq V(f)-V(g_e)$ is precisely the set of $f$ such that, for every pair $q,q'$ such that $V(f)<q$ and $V(g_e)\geq q'$, $0<d(f,g_e)<q-q'$, which is also Borel.

Therefore the function $F$ is Borel.  By $\mathsf{CFP}$ for Borel functions, there is a fixed point $f_*$ for $F$, and so $F(f_*)=f_*$.  Let $X=\{\beta: (f_*)_\beta\in[T_\beta]\}$.  Clearly if $\beta \in X$ then $T_\beta$ has a path.
Conversely, suppose that $\beta \notin X$, and suppose for a contradiction that $T_\beta$ has a path $h$ which is $\Sigma_{\prec}^{X\oplus f_* \oplus\sel T_n \ser_{n \in \Nb}}$.
Then we may define $g(\sel \gamma ,n\ser)$ by $h(n)$ if $\gamma =\beta $ and $f_*(\sel \gamma,n\ser)$ otherwise, so $g$ is $\Sigma_{\prec}^{X\oplus f_* \oplus\sel T_n\ser_{n \in \Nb}}$, and therefore $\Sigma_{\prec'}^{f_*\oplus\sel T_n\ser_{n \in \Nb}}$.
It follows that $g=g_e$ for some $e$.

Moreover, $0<d(f_*,g)\leq 2^{-\sel \beta,0\ser}\leq 2^{-\beta}=V(f_*)-V(g)$, which means that $F(f_*) = g_{e_0}$ for the least $e_0$ with $0<d(f_*,g_{e_0})\leq  V(f_*)-V(g_{e_0})$,
and so we must have $F(f_*)\neq f_*$, which is a contradiction.
Therefore $X$ and $f_*$ are the necessary witnesses.
\end{proof}

\appendix

\section{Remarks on Function Codes}\label{App}

In order to maintain some generality in the formalization of Caristi's theorem, we have based our presentation on Baire and Borel codes.
In this Appendix we establish how these codes relate to each other, as well as to other codes used in the literature.
 
First, we recall the standard coding of continuous functions used in e.g.~\cite{SimpsonSOSOA}.

\begin{Definition}[$\rca$; \cite{SimpsonSOSOA}*{Definition~II.6.1}]\label{def-continuous}
Let $\spc X = \widehat X$ and $\spc Y = \widehat Y$ be complete separable metric spaces.  A \emph{continuous partial function} $f\colon \spc X \to \spc Y$ is coded by a set $\Phi \subseteq \Nb \times X \times \Qb^{>0} \times Y \times \Qb^{>0}$ that satisfies the properties below.  Let us write $\ball ar \fcode \Phi \ball bq$ for $\exists n(\sel n, a, r, b, q \ser \in \Phi)$.  Then, for all $a, a' \in X$, all $q, q' \in \Qb$, and all $r, r' \in \Qb^{>0}$, $\Phi$ should satisfy:
\begin{enumerate}[label=({\sc cf} \arabic*)]
\item if $\ball{a}{r} \fcode \Phi \ball{b}{q}$ and $\ball{a}{r} \fcode \Phi \ball{b'}{q'}$, then $d(b,b') \leq q+q'$;

\item if $\ball{a}{r} \fcode \Phi \ball{b}{q}$ and $\ball{a'}{r'} \ballsub \ball{a}{r}$, then $\ball{a'}{r'}  \fcode \Phi \ball{b}{q}$;

\item if $\ball{a}{r} \fcode \Phi \ball{b}{q}$ and $\ball{b}{q} \ballsub \ball{b'}{q'}$, then $\ball{a}{r} \fcode \Phi \ball{b'}{q'}$.
\end{enumerate}

A point $x \in \spc X$ is in the domain of the function $f$ coded by $\Phi$ if, for every $\varepsilon > 0$, there are $\ball{a}{r}\fcode \Phi \ball{b}{s}$ such that $d(x, a) < r$ and $s < \varepsilon$. If $x \in  \dom(f)$, we define the value $f(x)$ to be the unique point $y \in \spc Y$ such that $d(y, b) \leq  s$ for all $\ball{a}{r} \fcode \Phi \ball{b}{s}$ with $d(x, a) < r$.
\end{Definition}

\begin{Lemma}[$\rca$]\label{lemmContBorel}
Let $\spc X,\spc Y$ be metric spaces.
A function $f\colon \spc X\to \spc Y$ can be coded as a continuous function in the sense of Definition~\ref{def-continuous} if and only if it can be coded as a Borel function where the preimage of every open ball is open.
\end{Lemma}

\begin{proof}[Proof sketch.]
If $\Phi$ is a code for a continuous function $f$ and $\ball bs$ is any open ball in $\spc Y$,
we can enumerate $f^{-1}[\ball bs] $
by enumerating $\ball ar$ if $\ball ar \fcode \Phi \ball {b'}{s'} $,
where $d_\spc Y(b,b') + s'<s$ (the latter is needed since $\ball ar \fcode \Phi \ball{b'}{s'} $ only guarantees $\ball ar \subseteq  f^{-1}[\overline{{\ball bs}}] $).
Conversely, if $\Psi$ is a Borel code for $f$, we note that for $\ball bs\subseteq \spc Y$, $f^{-1}[\ball bs] $ is represented in the form $ \bigcup_{i\in\mathbb N} \ball{a_i}{r_i}$.
We thus enumerate $\ball ar \fcode \Phi \ball bs $ if $a=a_i$ and $r=r_i$ for some $i$.
\end{proof}
Indeed, the proof of Lemma~\ref{lemmContBorel} is effective in the sense that there
exist (provably in $\rca$) $\Delta^0_1$ formulas which define Turing functionals for these conversions of codes, and thus any sequence of codes of functions in one way can be converted to the sequence of codes in the other way.

The results of~\cite{EkelandSelecta} were originally stated with respect to the following coding of lower semi-continuous functions, but as we will see, it is equivalent to our Baire representation for them.

\begin{Definition}[$\rca$; \cite{EkelandSelecta}*{Definition~4.1}]\label{def-lsc}
Let $\spc X$ be a complete separable metric space.  A \emph{lower semi-continuous partial function} $V\colon \spc X \to {\Rb}$ is coded by a set ${\Psi} \subseteq \Nb \times X \times \Qb^{>0} \times \Qb$ that satisfies the properties below.  Let $\ball{a}{r} \vcode \Psi q$ denote $\exists n(\la n, a, r, q \ra \in {\Psi})$.
Then $\Psi$ must satisfy that for all $a, a' \in X$, all $q, q' \in \Qb$, and all $r, r' \in \Qb^{>0}$, 
\begin{enumerate}[label=({\sc lsc} \arabic*)]

\item if $\ball{a}{r} \vcode \Psi q$ and $\ball{a'}{r'}  \ballsub \ball{a}{r}$, then $\ball{a'}{r'} \vcode \Psi q$, and

\item if $\ball{a}{r} \vcode \Psi q$ and $q' < q$, then $\ball{a}{r} \vcode \Psi q'$.

\end{enumerate}

A point $x \in \spc X$ is in the domain of the function $V$ coded by $\Psi$ if 
\begin{align*}
y = \sup \{q \in \Qb : (\exists \la a,r \ra \in X \times \Qb^{>0})(\ball{a}{r} \vcode \Psi q \andd d(x,a) < r)\}
\end{align*}
exists, in which case $V(x) = y$.
If $V$ has codomain $[0,\infty)$ (in the sense that $\ball{a}{r} \vcode \Psi 0$ for every $a,r$), we call $V$ a {\em potential.}
\end{Definition}

Let $\spc X$ be a complete separable metric space.  The idea behind Definition~\ref{def-lsc} is that $\Psi$ enumerates pairs $\la B_r(a), [q, \infty) \ra$ with the property that if $V$ is the function being coded by $\Psi$ and $x$ is in $B_r(a) \cap {
\rm dom}(V)$, then $V(x)$ is in $[q, \infty)$.

\begin{Lemma}\label{lemmLSCCompare}\
\begin{enumerate}
\item Over $\rca$, it is provable that a function $f$ has an lsc code as given by Definition~\ref{def-lsc} if and only if it is a pointwise increasing limit of continuous functions.

\item Over $\pica$, it is provable that for every potential $f$ and every $\alpha>0$, the $\alpha$-envelope of $f$ exists.
\end{enumerate}

\end{Lemma}

\begin{proof}[Proof sketch.]
The second claim is proven in~\cite{EkelandSelecta}, so we focus on the first.
First we approximate the indicator function of an open ball, $\chi_{\ball ar}$; for $\varepsilon \in (0,1)$, define $\chi_\varepsilon (x)$ to be $1$ if $x\in \ball a{r(1-\varepsilon)}$, $0$ if $x\notin \ball ar$, and otherwise $\chi_\varepsilon(x) = \nicefrac 1\varepsilon - \nicefrac{d(a,x)}{ \varepsilon r}$.
Then, if we define $\chi^n_{\ball ar}:= \chi_{2^{-n-1}}$, it is clear that $\chi^n_{\ball ar} \to \chi_{\ball ar}$ as $n\to\infty$.
If $\sel \sel B_i,q_i\ser : i<m\ser $ is a tuple of pairs consisting of an open ball and a positive rational, we may similarly approximate $\max_{i<m} q_i \chi_{B_i}$ by $\max_{i<m} q_i \chi^n_{B_i}$.
If we enumerate an lsc code $\Psi$ as $\{\sel B_i,q_i\ser: i\in \Nb\}$, we may then approximate the function $V$ coded by $\Psi$ by diagonally approximating $V$ as  $\max_{i<n} q_i \chi^n_{B_i}$.

Conversely, if $V = \lim_{n\to \infty}V_n$ where $\sel V_n:n\in\Nb\ser$ is pointwise increasing, we may define a code for $V$ by putting $B \vcode\Psi q$ if there is $n$ such that $V_n(x) > q$ for all $x\in B$, which may be extracted from the continuous code for $V_n$.
\end{proof}

\begin{Lemma}[$\atr$]\label{lemBaireisBorel}
  Every Baire function is Borel.
\end{Lemma}
\begin{proof}
  Let $f=f_\Xi$ be a Baire function coded by $\Xi$.  It suffices to show that for all $B_r(a)$, the inverse image $f^{-1}[B_r(a)]$ is a Borel set in a uniform way.  We show this by recursion on $\Xi$.
  The case where $f$ is continuous is given by Lemma \ref{lemmContBorel}, so we assume otherwise.

  Suppose we have $f=\lim_{n\rightarrow \infty}f_n$ and each $f_n$ is given as a Borel function.  Then we may set
\[f^{-1}[B_r(a)]=\bigcup_{r'<r}\bigcup_n\bigcap_{m\geq n}f_m^{-1}[B_{r'}(a)].\]
\end{proof}

In general it is not true that every Borel function is Baire, but for many specific spaces this does hold.
However, it does hold for {\em zero-dimensional spaces,} are spaces that have a basis consisting of clopen balls.

\begin{Lemma}[$\atr$]\label{thm:N_borel_is_baire}
  If $\mathcal{X}$ is zero-dimensional, then every Borel function from $\mathcal{X}$ to $\mathbb{N}$ is Baire.
\end{Lemma}

Before beginning the proof, it will be helpful to refine our definitions a bit.  Let $\mathcal{A}$ be the algebra generated by the basic clopen sets of $\mathcal{X}$---that is, the collection of sets generated from basic clopen sets by complements, finite unions, and finite intersections.  We can choose some encoding of the elements of $\mathcal{A}$ by natural numbers.

We next modify our notion of a Borel code slightly.  First, we restrict ourselves to codes where the levels alternate between unions and intersections---that is, which have the form $\bigcup_n\bigcap_m\bigcup_p\cdots$.  We can easily obtain a code with this property from one without by compressing runs of numbers with the same parity.  For example, when $\sigma{}^\frown\sel 2n\ser\in S$ is not a leaf, we can replace this node with nodes of the form
\[\sigma{}^\frown\sel 2 \sel 2n,2m_1,\ldots,2m_k \ser \ser\]
such that $\sigma{}^\frown\sel 2n,2m_1,\ldots,2m_k\ser\in S$, and similarly converting runs of odd numbers to a single odd number.

Second, we require that the top level be a union---that is, unless $S$ contains only a single leaf, we require that the unique $n$ so that $\sel n\ser\in S$ be odd.  This is easily arranged, because if $S$ does not have this property then we replace it with $\{\sel 1\ser^\frown \sigma: \sigma\in S\}$.

 Finally, we allow leaves to be labeled by elements of $\mathcal{A}$; this does not change what sets are Borel, though it can slightly reduce the complexity.  For purposes of the proof, we call these \emph{clean Borel codes}.  We note that, because leaves are labeled by elements of $\mathcal{A}$, a clean Borel code is not quite a Borel code.  Nonetheless, the basic properties of Borel codes hold for clean Borel codes with no changes to the proofs.

The advantage to clean Borel codes is that we can easily take finite unions and intersections without changing the complexity, and can take the complement while increasing the complexity by at most one (because we need to add an extra, trivial, union step at the root).

\begin{proof}[Proof of Lemma \ref{thm:N_borel_is_baire}]
We assume $f$ is given as a clean Borel function---that is, the code for $f$ gives us, for each $n$, a clean Borel code for $f^{-1}(n)$.

We will define a well-founded tree $U$ and, for each node $\sigma\in U$, a clean Borel function $f_\sigma$.  Additionally, we need to give some sort of ordinal bound on the Borel codes we use, which will ensure that they are getting simpler as we progress to larger nodes in $U$; we take an approach which is inefficient but relatively simple to describe.  Take the tree combining all the Borel codes for $f$---that is, the tree $\Upsilon$ with branches indexed by $\mathbb{N}$, and above each $n$, the Borel code $S^{n}$ for $f^{-1}(n)$.  Then $\Upsilon$ is a well-founded tree.

For each $\sigma$ and each $n$, we will have an assignment $\pi_{\sigma}^n$ from the tree $S_{\sigma}^n$ encoding $f^{-1}_\sigma(n)$ to $\Upsilon$ so that if $\tau\sqsubset\tau'$ then $\pi_{\sigma}^n(\tau')<_{{\text{\sc{kb}}}(\Upsilon)}\pi_{\sigma}^n(\tau)$.  We have $\sel\ser\in U$ and $f_{\sel\ser}=f$.  The map $\pi_{\sel\ser}^n$ is the inclusion of $S^{n}$ in $\Upsilon$.

If, for every $n$, $f_\sigma^{-1}(n)$ is presented as an element of $\mathcal{A}$ (that is, the clean Borel code for $f_\sigma^{-1}(n)$ is simply a single leaf labeled by an element of $\mathcal{A}$) then $\sigma$ is a leaf of $U$.

Otherwise, we will define a sequence of functions $f_{\sigma{}^\frown\sel s\ser}$ so that $\lim_{s\rightarrow\infty}f_{\sigma^\frown\sel s\ser}=f_\sigma$.
If any $f_\sigma^{-1}(n)=B\in\mathcal{A}$, we may replace it with $f_\sigma^{-1}(n)=\bigcup_{i=1}^1 B$, so we assume that, for all $n$, $f_\sigma^{-1}(n)=\bigcup_i S^n_i$. 
Suppose that either $\{\pi_{\sigma}^n(\sel\ser): n\in\Nb\}$ doesn't have a largest element or $S^n_i\in\mathcal{A}$ for any $i,n\in\Nb$. 
Then we define 
\[f_{\sigma^\frown\sel s\ser}^{-1}(n)=\bigcup_{i\leq s}S^n_i\]
for $n<s$, 
\[f_{\sigma^\frown\sel s\ser}^{-1}(s)=\bigcap_{n<s}\bigcap_{i\leq s}\overline{S^n_i}\]
and $f_{\sigma^\frown\sel s\ser}^{-1}(n)=\varnothing$ for $n>s$.  Then it is immediate that $\lim_{s\rightarrow\infty}f_{\sigma^\frown\sel s\ser}=f_\sigma$.  The functions $\pi_{\sigma^\frown\sel s\ser}^n$ can be defined in the obvious way by copying over the definitions of the $\pi_\sigma^n$ where possible and, in $\pi_{\sigma^\frown\sel s\ser}^n$, mapping the extra levels to successors.  In particular, $\sup_n \pi_{\sigma^\frown\sel s\ser}^n(\sel\ser)<\sup_n\pi_{\sigma}^n(\sel\ser)$.

Otherwise, $\{\pi_{\sigma}^n(\sel\ser): n\in\Nb\}$ has a largest element and at least one $f_\sigma^{-1}(n)$ has the form $\bigcup_i\bigcap_j S_{i,j}^n$; then, again, we may assume that every $f_\sigma^{-1}(n)$ has this form (by replacing $\bigcup_i B_i$ with $\bigcup_i \bigcap_{j=1}^1 B_i$).  For each $s$ and $n$, we set
\[f_{\sigma^\frown\sel s\ser}^{-1}(n)=\bigcup_{t_1}\left[\bigcap_{m\leq s}S^n_{t_1,m}\cap\bigcap_{\sel u_0,u_1 \ser < \sel n,t_1 \ser}\bigcup_{m\leq s}\overline{S^{u_0}_{u_1,m}}\right].\]
That is, $f_{\sigma^\frown\sel s\ser}(x)$ is chosen by finding the smallest pair $\sel n,t_1 \ser$ such that $x\in\bigcap_{m\leq s}S^n_{t_1,m}$.  Such a pair $\sel n,t_1 \ser$ clearly exists and is unique, so $f_{\sigma^\frown\sel s\ser}$ also defines a unique Borel function.  Since, for each $x$, there is an $n$ so that, for some $t_1$, $x\in\bigcap_m S^n_{t_1,m}$, there is some large enough finite $s$ so that $s\not\in\bigcup_{\sel u_0,u_1 \ser < \sel n,t_1 \ser}\bigcap_{m\leq s}S^{u_0}_{u_1,m}$, and therefore $f_{\sigma^\frown\sel s\ser}(x)=f_\sigma(x)$; in particular, the functions $f_{\sigma^\frown\sel s\ser}$ converge pointwise to $f_\sigma$.

We need to define the functions $\pi_{\sigma^\frown\sel s\ser}^n$, and we have to do this a bit carefully to make sure that $\sup_n\pi_{\sigma^\frown\sel s\ser}^n(\sel\ser)<\sup_n\pi_\sigma^n(\sel\ser)$.  For simplicity, assume that each $S^n_{t_1,m}=\bigcup_j S^n_{t_1,m,j}$.  (In the cases where this fails, $S^n_{t_1,m}\in\mathcal{A}$, and the coding is similar but simpler.)  Then we have, for each $t_1$,
\[\bigcap_{m\leq s}S^n_{t_1,m}\cap\bigcap_{\sel u_0,u_1 \ser < \sel n,t_1 \ser}\bigcup_{m\leq s}\overline{S^{u_0}_{u_1,m}}=\bigcup_{j_1,\ldots,j_s}\left[\bigcap_{m\leq s}S^n_{t_1,m,j_m}\cap\bigcap_{\sel u_0,u_1 \ser < \sel n,t_1 \ser}\bigcup_{m\leq s}\overline{S^{u_0}_{u_1,m}}\right],\]
and using the closure of clean Borel codes under finite unions and intersections, the interior of the union is a single clean Borel code.  Note that $\pi_\sigma^n$ maps the node corresponding to each $S^{u_0}_{u_1,m}$ (or $S^n_{t_1,m}$) to a level at least three nodes below $\sup_n \pi_{\sigma^\frown\sel s\ser}^n(\sel\ser)$ (there must be a node above for $\bigcap_j$, then one for $\bigcup_i$, then one for the root of the Borel code).  By mapping to level-wise maxima, we may arrange for $\pi_{\sigma^\frown\sel s\ser}^n$ to map the root of this intersection to a level at least three below $\sup_n \pi_{\sigma^\frown\sel s\ser}^n(\sel\ser)$ as well.  Then we can map the union over $t_1,j_1,\ldots,j_s$ to the level immediately above that, and the root to the level above that.  In particular, $\pi^n_{\sigma^\frown\sel s\ser}(\sel\ser)+1\leq\sup_n\pi^n_\sigma(\sel\ser)$, and therefore $\sup_n\pi^n_{\sigma^\frown\sel s\ser}(\sel\ser)<\sup_n\pi^n_\sigma(\sel\ser)$.

Since ${\text{\sc{kb}}}(\Upsilon)$ is well-ordered, the map from $U$ to $\Upsilon$ given by $\sigma\mapsto\sup_n\pi^n_{\sigma}(\sel\ser)$ shows that $U$ is well-founded.  The leaves $\sigma$ of $U$ are functions where each $f^{-1}_\sigma(n)\in \mathcal{A}$, and since each open ball is clopen, each element of $\mathcal{A}$ is open, so $f^{-1}_\sigma$ is continuous.  We may replace each leaf with a code for $f_\sigma$ as a continuous function, and we have therefore obtained a representation of $f$ as a Baire function.
\end{proof}

\begin{Lemma}[$\atr$]\label{thm:borel_is_baire}
  Every Borel function from $\Nb^\Nb$ to itself is Baire.
\end{Lemma}
\begin{proof}
First, given a Borel function $f$ coded by $\Upsilon$, let $f': \Nb^\Nb \times \Nb \to \Nb$ be the function which maps $\sel \Lambda,n\ser$ to the $n$-th position of $f(\Lambda)$.  We view $\Nb^\Nb \times \Nb$ as a metric space in which $d(\sel \Lambda,n\ser ,\sel \Lambda',n'\ser )=2$ iff $n\neq n'$ and $d(\sel \Lambda,n \ser ,\sel \Lambda',n\ser )=d(\Lambda,\Lambda')\leq 1$.
Then $f'$ is also Borel: the inverse image of $a$ is the union over all sets $U_{\Upsilon_{ \sel \sigma,2^{-|\sigma| }\ser }}\times\{n\}$, where $n\in\Nb$ and $\sigma\in\Nb^{<\Nb}$ such that $\sigma(n) = a$.

By the previous lemma, $f'$ is also Baire.  Taking the (continuous) leaves $f'_\sigma : \Nb^\Nb \times \Nb \to \Nb$, we can define $f_\sigma : \Nb^\Nb \to \Nb^\Nb$ by $f_\sigma(\Lambda)=\sel f'_\sigma(\sel \Lambda,n\ser) : n \in \Nb \ser$.  This function is still continuous and respects the limits of the $f'_\sigma$ sequence, so the same tree gives a representation of $f$ as a Baire function.
\end{proof}

\section*{Funding}
Fernández-Duque was supported by the FWO-FWF Lead Agency grant G030620N (FWO)\slash I4513N (FWF).
Shafer was supported by the \emph{Fonds voor Wetenschappelijk Onderzoek -- Vlaanderen} Pegasus program, the John Templeton Foundation grant ID 60842 \emph{A new dawn of intuitionism: mathematical and philosophical advances} and by EPSRC grant EP/T031476/1 \emph{Reverse mathematics of general topology}.  The opinions expressed in this work are those of the authors and do not necessarily reflect the views of the John Templeton Foundation.
Towsner was partially supported by NSF grant DMS-2054379.
Yokoyama was partially supported by  JSPS KAKENHI grant number 19K03601 and 21KK0045.

\section*{Acknowledgments}

David Fernández-Duque would like to thank Carlos Bosch Giral, César Luis García García, Claudia Gómez Wulschner, Rigoberto Vera Mendoza, and the other attendees of the ITAM Analysis seminar, for introducing him to Caristi's theorem and for many fruitful discussions.


\begin{bibdiv}
\begin{biblist}

\bib{Caristi}{article}{
      author={Caristi, J.},
       title={Fixed point theorems for mappings satisfying inwardness
  conditions},
        date={1976},
     journal={Transactions of the American Mathematical Society},
      volume={215},
       pages={241\ndash 251},
}

\bib{DFSW}{article}{
      author={Dzhafarov, D.},
      author={Flood, S.},
      author={Solomon, R.},
      author={Westrick, L.},
       title={Effectiveness for the dual {R}amsey theorem},
        date={2021},
        ISSN={0029-4527},
     journal={Notre Dame J. Form. Log.},
      volume={62},
      number={3},
       pages={455\ndash 490},
         url={https://doi.org/10.1215/00294527-2021-0024},
      review={\MR{4323042}},
}

\bib{ekeland1974}{article}{
      author={Ekeland, I.},
       title={On the variational principle},
        date={1974},
     journal={Journal of Mathematical Analysis and Applications},
      volume={47},
      number={2},
       pages={324\ndash 353},
}

\bib{EkelandSelecta}{article}{
      author={Fernández-Duque, D.},
      author={Shafer, P.},
      author={Yokoyama, K.},
       title={Ekeland's variational principle in weak and strong systems of
  arithmetic},
        date={2020},
     journal={Selecta Mathematica},
      volume={26},
      number={68},
}

\bib{Freund2020}{article}{
      author={Freund, A.},
       title={What is effective transfinite recursion in reverse mathematics?},
        date={2020},
        ISSN={0942-5616},
     journal={MLQ Math. Log. Q.},
      volume={66},
      number={4},
       pages={479\ndash 483},
         url={https://doi.org/10.1002/malq.202000042},
      review={\MR{4267051}},
}

\bib{CaristiApps}{article}{
      author={Khojasteh, F.},
      author={Karapinar, E.},
      author={Khandani, H.},
       title={Some applications of Caristi’s fixed point theorem in metric
  spaces},
        date={2016},
     journal={Fixed Point Theory and Applications},
      volume={16},
      number={1},
}

\bib{Kirk2003}{article}{
      author={Kirk, W.~A.},
       title={Transfinite methods in metric fixed-point theory},
        date={2003},
     journal={Abstract and Applied Analysis},
      volume={2003},
      number={5},
       pages={311\ndash 324},
}

\bib{MR1428011}{incollection}{
      author={Marcone, A.},
       title={On the logical strength of {N}ash-{W}illiams' theorem on
  transfinite sequences},
        date={1996},
   booktitle={Logic: from foundations to applications ({S}taffordshire, 1993)},
      series={Oxford Sci. Publ.},
   publisher={Oxford Univ. Press, New York},
       pages={327\ndash 351},
      review={\MR{1428011}},
}

\bib{Peng2017}{article}{
      author={Peng, W.},
      author={Yamazaki, T.},
       title={Two kinds of fixed point theorems and reverse mathematics},
        date={2017},
     journal={Mathematical Logic Quarterly},
      volume={63},
      number={5},
       pages={454\ndash 461},
}

\bib{PriessCrampe2000}{article}{
      author={Priess-Crampe, S.},
      author={Ribenboim, P.},
       title={Ultrametric spaces and logic programming},
        date={2000},
        ISSN={0743-1066},
     journal={The Journal of Logic Programming},
      volume={42},
      number={2},
       pages={59\ndash 70},
}

\bib{priess-crampe2011}{article}{
      author={Priess-Crampe, S.},
      author={Ribenboim, P.},
       title={Ultrametric dynamics},
        date={2011},
     journal={Illinois J. Math.},
      volume={55},
      number={1},
       pages={287\ndash 303},
         url={https://doi.org/10.1215/ijm/1355927037},
}

\bib{SavateevS21}{article}{
      author={Savateev, Y.},
      author={Shamkanov, D.~S.},
       title={Non-well-founded proofs for the Grzegorczyk modal logic},
        date={2021},
     journal={Rev. Symb. Log.},
      volume={14},
      number={1},
       pages={22\ndash 50},
         url={https://doi.org/10.1017/S1755020319000510},
}

\bib{SimpsonSOSOA}{book}{
      author={Simpson, S.~G.},
       title={{Subsystems of Second Order Arithmetic}},
     edition={Second},
      series={Perspectives in Logic},
   publisher={Cambridge University Press, Cambridge; Association for Symbolic
  Logic, Poughkeepsie, NY},
        date={2009},
        ISBN={978-0-521-88439-6},
         url={http://dx.doi.org/10.1017/CBO9780511581007},
      review={\MR{2517689 (2010e:03073)}},
}

\bib{Towsner2013}{article}{
      author={Towsner, H.},
       title={Partial impredicativity in reverse mathematics},
        date={2013},
        ISSN={0022-4812},
     journal={J. Symbolic Logic},
      volume={78},
      number={2},
       pages={459\ndash 488},
         url={http://projecteuclid.org/euclid.jsl/1368627060},
      review={\MR{3145191}},
}

\bib{MR1197207}{article}{
      author={Yu, X.},
       title={Riesz representation theorem, {B}orel measures and subsystems of
  second-order arithmetic},
        date={1993},
        ISSN={0168-0072},
     journal={Ann. Pure Appl. Logic},
      volume={59},
      number={1},
       pages={65\ndash 78},
         url={http://dx.doi.org/10.1016/0168-0072(93)90232-3},
      review={\MR{1197207}},
}

\end{biblist}
\end{bibdiv}
\end{document}